\RequirePackage{fix-cm}
\documentclass[smallextended]{svjour3}       
\smartqed  

\usepackage{mathrsfs,url,color}
\RequirePackage[colorlinks,citecolor=blue,urlcolor=blue]{hyperref}

\usepackage{graphicx}
\usepackage{amsmath,amsfonts,epsfig,wasysym,enumerate,amssymb}

\newtheorem{thm}{Theorem}[section]

\newtheorem{lem}[thm]{Lemma}

\newtheorem{cor}[thm]{Corollary}

\newtheorem{exm}[thm]{Example}
\newtheorem{rem}[thm]{Remark}

\makeatletter \@addtoreset{equation}{section}
\renewcommand{\thesection}{\arabic{section}}

\usepackage{latexsym}
\def\R{\mathbb{R}}

\def\P{\mathbb{P}}
\def\E{\mathbb{E}}
\def\var{\operatorname{Var}}		
\def\cov{\operatorname{Cov}}		
\def\dx{\operatorname{d}}

\begin{document}
	
	\title{On a new concept of stochastic domination and the laws of large numbers}

	\titlerunning{A new concept of stochastic domination and the laws of large numbers}        
	
	\author{L\^{e} V\v{a}n Th\`{a}nh}
	
	\institute{	L\^{e} V\v{a}n Th\`{a}nh \at 
		Department of Mathematics, Vinh University, Nghe An, Vietnam\\
		\email{levt@vinhuni.edu.vn}
	}
	
	\date{}


	\maketitle
	
	\begin{abstract}			
		Consider a sequence of positive integers $\{k_n,n\ge1\}$, 
		and an array of nonnegative real numbers $\{a_{n,i},1\le i\le k_n,n\ge1\}$
		satisfying $\sup_{n\ge 1}\sum_{i=1}^{k_n}a_{n,i}=C_0\in (0,\infty).$
		This paper introduces the concept of $\{a_{n,i}\}$-stochastic domination.
		We develop some techniques concerning this concept and apply them
		to remove an assumption in a strong law of large numbers of Chandra and Ghosal [Acta. Math. Hungarica, 1996].
		As a by-product, a considerable extension of a recent result of Boukhari [J. Theoret. Probab., 2021]
		is established and proved by a different method. The results on laws of large
		numbers are new
		even when the summands are independent. Relationships between
		the concept of $\{a_{n,i}\}$-stochastic domination and the concept of $\{a_{n,i}\}$-uniform integrability are presented.
		Two open problems are also discussed.
		
	\end{abstract}

	\section{Introduction and Motivation}\label{sec:intro}
	
	Let $1\le p<2$. The classical Marcinkiewicz--Zygmund strong law of large numbers (SLLN)
	states that for a sequence $\{X_n,n\ge1\}$ of independent identically distributed mean zero
	random variables, condition $\E(|X_1|^p)<\infty$ is
	necessary and sufficient for
	\begin{equation}\label{chandra05}
		\lim_{n\to\infty}\dfrac{\sum_{i=1}^nX_i}{n^{1/p}}=0 \text{ almost surely (a.s.).}
	\end{equation}
	Now, let us recall a weak dependence structure introduced in Chandra and Ghosal \cite{chandra1996extensions} as follows.
	A sequence of random variables $\{X_n,n\ge1\}$
	is said to be \textit{asymptotically almost negatively associated} (AANA) if there exists
	a sequence of nonnegative real numbers $\{q_n,n\ge1\}$ with $\lim_{n\to\infty}q_n=0$ such that
	\[\cov(f(X_n),g(X_{n+1},\ldots,X_{n+k}))\le q_n\left(\var(f(X_n))\var(g(X_{n+1},\ldots,X_{n+k}))\right)^{1/2},\]
	for all $n\ge1, k\ge 1$ and for all coordinatewise nondecreasing continuous functions $f$ and $g$
	provided the right side of the above inequality is finite. The $q_n,n\ge1$ are called \textit{mixing coefficients}. The starting point of the current investigation
	is the following SLLN established by Chandra and Ghosal \cite{chandra1996extensions}.
	
	\begin{thm}[Chandra and Ghosal \cite{chandra1996extensions}]\label{thm.chandra} 
		Let $1\le p<2$ and let $\{X_n,n\ge1\}$ be a sequence of AANA mean zero random variables 
		with the sequence of mixing coefficients
		satisfying $\sum_{n=1}^{\infty}q_{n}^2<\infty$. Let
		\[G(x)=\sup_{n\ge1}\dfrac{1}{n}\sum_{i=1}^{n}\P(|X_i|>x),\ x\in\R.\]
		If
		\begin{equation}\label{chandra01}
			\int_{0}^{\infty}x^{p-1} G(x)\dx x<\infty,
		\end{equation}
		and
		\begin{equation}\label{chandra03}
			\sum_{n=1}^{\infty}\P(|X_n|^p>n)<\infty,
		\end{equation}
		then the Marcinkiewicz--Zygmund SLLN \eqref{chandra05} is obtained.
	\end{thm}
	The above result of Chandra and Ghosal \cite{chandra1996extensions} weakens the assumptions
	in the classical Marcinkiewicz--Zygmund SLLN not only
	by considering a weak dependence structure, but also by relaxing
	the identical distribution condition. We refer to \eqref{chandra01}
	and \eqref{chandra03} as the Chandra--Ghosal conditions.
	It is clear that for a sequence $\{X_n,n\ge1\}$ of random variables with a common
	law, \eqref{chandra01} and \eqref{chandra03} are equivalent since each of them is equivalent to $\E(|X_1|^p)<\infty$.
	This leads to a natural question in this context is whether
	\eqref{chandra01} or \eqref{chandra03} can be removed.
	The current work is an attempt to answer this
	question. More precisely, we shall prove the following theorem.
	\begin{thm}\label{thm.lln.main1}
		Theorem \ref{thm.chandra} holds without Condition \eqref{chandra03}.
	\end{thm}
	
	To prove Theorem \ref{thm.lln.main1}, we develop some results concerning a new concept of
	stochastic domination which leads to the concept of the Ces\`{a}ro
	stochastic domination as a
	particular case. 
	
	Let $\{k_n,n\ge1\}$ be a sequence of positive integers. An array $\{X_{n,i},1\le i\le k_n,n\ge1\}$ of
	random variables is said to be 
	\textit{stochastically dominated} 
	by a random variable $X$ if
	\begin{equation}\label{eq.stoch.domi.03}
		\sup_{1\le i\le k_n, n\ge 1}\P(|X_{n,i}|>x)\le \P(|X|>x), \ \text{ for all } x\in\mathbb{R}.
	\end{equation}
	This concept was extended to the concept of the so-called Ces\`{a}ro
	stochastic domination by Gut \cite{gut1992complete} as follows.
	An array $\{X_{n,i},1\le i\le k_n,n\ge1\}$ of
	random variables is said to be 
	\textit{stochastically dominated in the Ces\`{a}ro sense} (or \textit{weakly mean dominated})
	by a random variable $X$ if
	\begin{equation}\label{eq.stoch.domi.05}
		\sup_{n\ge 1} \dfrac{1}{k_n}\sum_{i=1}^{k_n}\P(|X_{n,i}|>x)\le C\P(|X|>x), \ \text{ for all } x\in\mathbb{R},
	\end{equation}
	where $C>0$ is a constant. 
	It was shown by Gut \cite[Example 2.1]{gut1992complete} that
	\eqref{eq.stoch.domi.05} is strictly weaker than \eqref{eq.stoch.domi.03}.
	
	We will now introduce a new concept of stochastic domination. Let $\{a_{n,i},1\le i\le k_n,n\ge1\}$ be an array of nonnegative real numbers
	satisfying
	\begin{equation}\label{eq.weight01}
		\sup_{n\ge 1}\sum_{i=1}^{k_n}a_{n,i}=C_0\in (0,\infty).
	\end{equation}
	An array $\{X_{n,i},1\le i\le k_n,n\ge1\}$ of
	random variables is said to be 
	\textit{$\{a_{n,i}\}$-stochastically dominated} 
	by a random variable $X$ if
	\begin{equation}\label{eq.stoch.domi.07}
		\sup_{n\ge 1} \sum_{i=1}^{k_n}a_{n,i}\P(|X_{n,i}|>x)\le C_0\P(|X|>x), \ \text{ for all } x\in\mathbb{R}.
	\end{equation} 
	In view of Gut's definition in \eqref{eq.stoch.domi.05}, one may be tempted to give
	an apparently weaker definition of $\{X_{n,i},1\le i\le k_n,n\ge1\}$ being $\{a_{n,i}\}$-stochastically dominated
	by a random variable $Y$, namely that
	\begin{equation}\label{eq.stoch.domi.08}
		\sup_{n\ge 1} \sum_{i=1}^{k_n}a_{n,i}\P(|X_{n,i}|>x)\le C\P(|Y|>x),\ \text{ for all }x\in\mathbb{R},
	\end{equation}	
	for some finite constant $C> 0$. However, it will be shown in Theorem \ref{thm.character.for.stochastic.domination.2} that
	\eqref{eq.stoch.domi.07} and \eqref{eq.stoch.domi.08} are indeed equivalent.
	Therefore, concerning Gut's definition of the Ces\`{a}ro stochastic domination, we can simply choose $C=1$
	in \eqref{eq.stoch.domi.05}.	
	If $a_{n,i}=1/k_n, 1\le i\le k_n,\ n\ge1,$
	then it is obvious that $C_0=1$, and the concept of $\{a_{n,i}\}$-stochastic domination reduces to
	the concept of stochastic domination in the Ces\`{a}ro sense.
	
	If $0<p<1$ and $\{X_n,n\ge1\}$ is a sequence of
	random variables satisfying \eqref{chandra01} and \eqref{chandra03},
	then the Marcinkiewicz--Zygmund SLLN \eqref{chandra05}
	is valid irrespective of any dependence structure (see Remark 3 in Chandra and Ghosal \cite{chandra1996extensions}).
	Boukhari \cite{boukhari2021weak} recently
	used techniques from martingales theory to prove that a similar result holds true for the weak law of
	large numbers (WLLN) for maximal partial sums with general normalizing sequences. The tools developed in this paper also allow
	us to establish an extension of Theorem 1.2 of Boukhari \cite{boukhari2021weak}. A special case of our WLLNs in Section \ref{sec:proofs}
	is the following theorem.
	
	\begin{thm}\label{thm.lln.main3}
		Let $\{X_n,n\ge1\}$ be a sequence of random variables, $G(\cdot)$
 as in Theorem \ref{thm.chandra} and let
		$\{b_n,n\ge1\}$ be a nondecreasing sequence of positive real numbers such that
		\begin{equation}\label{wl01}
			\sum_{i=1}^{n}\dfrac{b_{i}}{i^2}=O\left(\dfrac{b_{n}}{n}\right).
		\end{equation}
		If
		\begin{equation}\label{wl03}
			\lim_{k\to\infty}kG(b_k)=0,
		\end{equation}
		then the WLLN
		\begin{equation}\label{wl05}
			\dfrac{1}{b_n}\max_{j\le n}\left|\sum_{i=1}^j X_i\right|\overset{\P}{\to} 0 \text{ as }n\to\infty
		\end{equation}
		is obtained.
	\end{thm}
	
	\begin{rem}{\rm 
			Boukhari \cite[Theorem 1.2]{boukhari2021weak} proved Theorem
			\ref{thm.lln.main3} under a stronger condition that
			the sequence $\{X_n,n\ge1\}$ is stochastically dominated by a random variable $X$
			satisfying
			\begin{equation*}\label{boukhari01}
				\lim_{k\to\infty}k\P(|X|>b_k)=0.
			\end{equation*}
			An example in Section \ref{sec:proofs} shows that for $0<p<1$ and $b_n=n^{1/p}$,
			there exists a sequence of random variables $\{X_n,n\ge1\}$ with no stochastically dominating random variable,
			but
			\eqref{wl03} is satisfied and therefore the WLLN \eqref{wl05} is valid.
			Our proof of Theorem \ref{thm.lln.main3} is simpler than that of Theorem 1.2 
			of Boukhari \cite{boukhari2021weak} in the sense that we do not use the Doob maximal inequality
			for martingales as was done in Boukhari \cite{boukhari2021weak}.
			The WLLN for dependent random variables and random vectors was also studied in \cite{hien2015weak,kruglov2011generalization,rosalsky2009weak}, among others.
		}
	\end{rem}
	
	The rest of the paper is organized as follows. In Section \ref{sec:stochastic_domination},
	we prove the equivalence between the definitions of $\{a_{n,i}\}$-stochastic domination given in \eqref{eq.stoch.domi.07} and \eqref{eq.stoch.domi.08}.
	It is also shown that certain bounded moment
	conditions on an array of random variables $\{X_{n,i},1\le i\le k_n,n\ge1\}$
	can accomplish the concept of $\{X_{n,i},1\le i\le k_n,n\ge1\}$ being $\{a_{n,i}\}$-stochastically dominated. Section \ref{sec:relation}
	discusses about relationships between the concept of $\{a_{n,i}\}$-stochastic domination and
	the concept of $\{a_{n,i}\}$-uniform integrability.
	Strong and weak laws of large numbers for triangular arrays of random variables
	are presented in Section \ref{sec:proofs}. From these general results,
	Theorems \ref{thm.lln.main1} and \ref{thm.lln.main3} follow.  Section \ref{sec:open}
	contains further remarks and two open problems.\\

	\noindent \textit{Notation:} Throughout this paper, $\{k_n,n\ge1\}$ is assumed to be a sequence of
	positive integers. For a set $A$, $\mathbf{1}(A)$ denotes the indicator function of $A$.
	For $x\ge0$, let $\log x$ denote the logarithm base $2$ of
	$\max\{2,x\}$. For $x\ge0$ and for a fixed positive integer $\nu$, let
	\begin{equation}\label{notation1}
		\log_\nu(x):=(\log x)(\log\log x)\ldots (\log\cdots\log x),
	\end{equation}
	and
	\begin{equation}\label{notation2}
		\log_{\nu}^{(2)}(x):=(\log x)(\log\log x)\ldots (\log\cdots\log x)^{2},
	\end{equation}
	where in both \eqref{notation1} and \eqref{notation2}, there are $\nu$ factors.
	For example, $\log_2(x)=(\log x)(\log\log x)$, $\log_{3}^{(2)}(x)=(\log x)(\log\log x)(\log\log\log x)^{2}$, and so on.
	
	\section{On the concept of $\{a_{n,i}\}$-stochastic domination}\label{sec:stochastic_domination}
	In this section, we employ some properties of slowly varying functions as well as techniques in Rosalsky and Th\`{a}nh \cite{rosalsky2021note} to prove
	some results on
	the	concept of $\{a_{n,i}\}$-stochastic domination.
	We note that all results in Sections \ref{sec:stochastic_domination} and \ref{sec:relation}
	are stated for a triangular array  $\{X_{n,i},1\le i\le k_n,n\ge1\}$ of random variables but they still hold for
	a sequence of random variables $\{X_n,n\ge1\}$ by considering $X_{n,i}=X_i$, $1\le i\le k_n,n\ge1$.
	
	The following theorem is a simple result and its proof is similar to that of Theorem 2.1 of Rosalsky and Th\`{a}nh \cite{rosalsky2021note}.
	It plays a useful role in proving the laws of large numbers in Section \ref{sec:proofs}.
	
	\begin{thm}\label{thm.character.for.stochastic.domination.1}
		Let $\{X_{n,i},1\le i\le k_n,n\ge1\}$ be an array of random variables, let
		$\{a_{n,i},1\le i\le k_n,n\ge1\}$ be an array of 
		nonnegative
		real numbers satisfying \eqref{eq.weight01}
		and	let \[F(x)=1- \dfrac{1}{C_0}\sup\limits_{n\ge1}\sum_{i=1}^{k_n}a_{n,i}\P(|X_{n,i}|>x),\ x\in\R.\]
		Then $F(\cdot)$ is the distribution function of a random variable $X$ if and only if $\lim_{x\to \infty}F(x)=1.$
		In such a case, $\{X_{n,i},1\le i\le k_n,n\ge1\}$ is $\{a_{n,i}\}$-stochastically dominated by $X$.
	\end{thm}
	
	\begin{proof} It is clear that $F(\cdot)$ is nondecreasing, and
		\[\lim\limits_{x\to-\infty}F(x)=1-\dfrac{1}{C_0}\sup_{n\ge 1}\sum_{i=1}^{k_n}a_{n,i}=0.\]
		Let $\varepsilon >0$ be arbitrary. For $a\in\mathbb{R}$, let $n_0\ge 1$ be such that
		\[\dfrac{1}{C_0}\sum_{i=1}^{k_{n_0}}a_{n_0,i}\P(|X_{n_0,i}|>a)>\dfrac{1}{C_0}\sup_{n\ge1}\sum_{i=1}^{k_n}a_{n,i}\P(|X_{n,i}|>a)-\varepsilon/2,\]
		or equivalently,
		\begin{equation}\label{df01}
			1-\dfrac{1}{C_0}\sum_{i=1}^{k_{n_0}}a_{n_0,i}\P(|X_{n_0,i}|>a)<F(a)+\varepsilon /2.
		\end{equation}
		Since the function
		\[x\mapsto \dfrac{1}{C_0}\sum_{i=1}^{k_{n_0}}a_{n_0,i}\P(|X_{n_0,i}|>x),\ x\in\R, \]
		is nonincreasing and right continuous, there exists $\delta>0$ such that
		\[-\varepsilon /2<\dfrac{1}{C_0}\sum_{i=1}^{k_{n_0}}a_{n_0,i}\P(|X_{n_0,i}|>x)-\dfrac{1}{C_0}\sum_{i=1}^{k_{n_0}}a_{n_0,i}\P(|X_{n_0,i}|>a)\le 0 \text{ for  all $x$ such that } 0\le x-a<\delta.\]
		Therefore, for $x$ satisfying $0\le x-a<\delta$, we have 
		\begin{align*}
			F(x)-\varepsilon 
			&=1-\dfrac{1}{C_0}\sup\limits_{n\ge1}\sum_{i=1}^{k_n}a_{n,i}\P(|X_{n,i}|>x)-\varepsilon \\
			&\le 1- \dfrac{1}{C_0}\sum_{i=1}^{k_{n_0}}a_{n_0,i}\P(|X_{n_0,i}|>x)-\varepsilon \\
			&<1- \dfrac{1}{C_0}\sum_{i=1}^{k_{n_0}}a_{n_0,i}\P(|X_{n_0,i}|>a)-\varepsilon /2\\
			&<F(a)\ \text{ (by \eqref{df01})}
		\end{align*}
		and so $|F(x)-F(a)|<\varepsilon$. Thus $\lim\limits_{x\to a^+}F(x)=F(a)$. Since $a\in \R$ is arbitrary, this implies that $F$ is right
		continuous on $\R$. Since $F(\cdot)$ is nondecreasing, right continuous and $\lim\limits_{x\to-\infty}F(x)=0$,
		it is the distribution function of a random variable $X$ if and only if
		$\lim\limits_{x\to \infty}F(x)=1$. 
		By definition of $F(\cdot)$, $\{X_{n,i},1\le i\le k_n,n\ge1\}$
		is $\{a_{n,i}\}$-stochastically dominated by $X$. 
		\qed
	\end{proof}

	The next theorem establishes the equivalence between the definitions of  $\{a_{n,i}\}$-stochastic domination given in \eqref{eq.stoch.domi.07} and \eqref{eq.stoch.domi.08}.
	A similar result concerning the concept of stochastic domination was proved by Rosalsky and Th\`{a}nh \cite{rosalsky2021note}.

	\begin{thm}\label{thm.character.for.stochastic.domination.2}
		Let $\{X_{n,i},1\le i\le k_n,n\ge1\}$ be an array of random variables and let
		$\{a_{n,i},1\le i\le k_n,n\ge1\}$ be an array of
		nonnegative
		real numbers satisfying \eqref{eq.weight01}. 
		Then there exists a random variable $X$ satisfying \eqref{eq.stoch.domi.07}
		if and only if there exist
		a random variable $Y$ and a finite constant $C>0$ satisfying \eqref{eq.stoch.domi.08}.
		
		Moreover, (i) if $g:[0,\infty)\to [0,\infty)$ is a measurable function with $g(0)=0$ which is bounded on $[0,A]$ and differentiable on $[A,\infty)$ for some $A\ge 0$
		or (ii) if $g:[0,\infty)\to [0,\infty)$ is a continuous function which is eventually nondecreasing with $\lim_{x\to\infty}g(x)=\infty$,
		then the condition
		$\E(g(|Y|))<\infty$ where $Y$ is as in \eqref{eq.stoch.domi.08} implies that $\E(g(|X|))<\infty$ where $X$ is as in \eqref{eq.stoch.domi.07}.
	\end{thm}
	\begin{proof}
		The necessity half is immediate by taking $Y=X$ and $C=C_0$.
		Conversely, if there exist a nonnegative random variable 
		$Y$ and a finite constant $C>0$ satisfying \eqref{eq.stoch.domi.08}, then
		\begin{equation*}
			\begin{split}
				\lim_{x\to\infty}\sup_{n\ge 1} \sum_{i=1}^{k_n}a_{n,i}\P(|X_{n,i}|>x)\le C\lim_{x\to\infty}\P(|Y|>x)=0,
			\end{split}
		\end{equation*}
		and so by Theorem \ref{thm.character.for.stochastic.domination.1}, there exists a random variable $X$ with distribution function 
		\[F(x)=1-\dfrac{1}{C_0}\sup_{n\ge 1} \sum_{i=1}^{k_n}a_{n,i}\P(|X_{n,i}|>x),\ x\in\R.\]
		This implies
		\begin{equation*}
			\begin{split}
				\sup_{n\ge 1} \sum_{i=1}^{k_n}a_{n,i}\P(|X_{n,i}|>x)=C_0(1-F(x))=C_0\P(X>x),\ x\in\R
			\end{split}
		\end{equation*}
		thereby verifying \eqref{eq.stoch.domi.07}.
		
		The rest of the proof proceeds in a similar manner as that of Theorem 2.4 (i) and (ii)
		in Rosalsky and Th\`{a}nh \cite{rosalsky2021note}. The details will be omitted. 
		\qed
	\end{proof}
	
	The following result is a direct consequence of Theorem \ref{thm.character.for.stochastic.domination.2}
	by choosing $a_{n,i}=1/k_n$ for all $1\le i\le k_n$, $n\ge1$. 
	It says that in the Ces\`{a}ro stochastic domination definition, \eqref{eq.stoch.domi.05} can be simplified to 
	\begin{equation}\label{eq.Gut05.stoch.domi}
		\sup_{n\ge 1} \dfrac{1}{k_n}\sum_{i=1}^{k_n}\P(|X_{n,i}|>x)\le \P(|Y|>x), \ \text{ for all } x\in\mathbb{R},
	\end{equation}
	for some random variable $Y$ (surprisingly, this was not
	noticed by Gut \cite{gut1992complete}).
	\begin{cor}\label{cor.compare.Gut}
		Let $\{X_{n,i},1\le i\le k_n,n\ge1\}$ be an array of random variables.
		Then there exist a random variable $X$ and a finite constant $C>0$ satisfying \eqref{eq.stoch.domi.05}
		if and only if there exists
		a random variable $Y$ satisfying \eqref{eq.Gut05.stoch.domi}.
		Moreover, if $g$ is a measurable function satisfying assumptions in Theorem \ref{thm.character.for.stochastic.domination.2}, 
		then the condition  $\E(g(|X|))<\infty$ where $X$ is as in \eqref{eq.stoch.domi.05}
		implies that $\E(g(|Y|))<\infty$ where $Y$ is as in \eqref{eq.Gut05.stoch.domi}.
	\end{cor}
	
	By using integration by parts, we have from \eqref{eq.Gut05.stoch.domi} that for all $r>0$ and $x\ge0$
	\begin{equation}\label{eq.Gut07.stoch.domi}
		\sup_{n\ge 1} \dfrac{1}{k_n}\sum_{i=1}^{k_n}\E\left(|X_{n,i}|^r\mathbf{1}(|X_{n,i}|\le x)\right)\le \E\left(|Y|^r\mathbf{1}(|Y|\le x)\right)+x^{r}\P(|Y|>x),
	\end{equation}	
	and
	\begin{equation}\label{eq.Gut09.stoch.domi}
		\sup_{n\ge 1} \dfrac{1}{k_n}\sum_{i=1}^{k_n}\E\left(|X_{n,i}|^r\mathbf{1}(|X_{n,i}|>x)\right)\le \E\left(|Y|^r\mathbf{1}(|Y|>x)\right).
	\end{equation}
	We will use \eqref{eq.Gut07.stoch.domi} and \eqref{eq.Gut09.stoch.domi} in our proofs without further mention.
	
	The following consequence of Theorem \ref{thm.character.for.stochastic.domination.1}
	is also useful in proving the laws of large numbers.
	
	\begin{cor}\label{cor.character.for.stochastic.domination.Gut1}
		Let $\{X_{n,i},1\le i\le k_n,n\ge1\}$ be an array of random variables
		and	let \[F(x)=1-\sup\limits_{n\ge1}\dfrac{1}{k_n}\sum_{i=1}^{k_n}\P(|X_{n,i}|>x),\ x\in\R. \]
		Then $F$ is the distribution function of a random variable $X$  if and only if 
		$\lim_{x\to \infty}F(x)=1.$
		In such a case, $\{X_{n,i},1\le i\le k_n,n\ge1\}$ is stochastically dominated in the Ces\`{a}ro sense by $X$.
	\end{cor}

	In view of \eqref{eq.Gut05.stoch.domi}, if
	$\{X_{n,i},1\le i\le k_n,n\ge1\}$ is stochastically dominated in the Ces\`{a}ro sense by a random variable $Y$, 
	then
	it is $\{a_{n,i}\}$-stochastically dominated by $Y$ with $a_{n,i}=1/k_n$, $1\le i\le k_n$, $n\ge1$.
	The following example shows that the concept of $\{a_{n,i}\}$-stochastic domination is strictly 
	weaker than the concept of stochastic domination in the Ces\`{a}ro sense.
	
	\begin{exm}\label{exm.stoch.domi}{\rm 
			Let $k_n\equiv n$ and let $m_n$ be the greatest integer number
			which is less than or equal to $n/2$. Let $a_{n,i}=1/m_n$ for $1\le i\le m_n,\ n\ge 2$ and $a_{n,i}=1/n^2$ for $m_n<i\le n, n\ge1$.
			Let $\{X_{n,i},1\le i\le n,n\ge1\}$ be an array of random variables such that 
			\[\P(X_{n,i}=-1)=\P(X_{n,i}=1)=1/2,\ 1\le i\le m_n, \ n\ge 2\]
			and
			\[\P(X_{n,i}=-n)=\P(X_{n,i}=n)=1/2,\ m_n< i\le n, \ n\ge 1.\]
			Then, for $x\ge 1$, we have
			\[\dfrac{1}{n}\sum_{i=1}^n \P(|X_{n,i}|>x)=
			\begin{cases}
				0 & \text{ if }n\le x,\\
				\dfrac{n-m_n}{n} & \text{ if }n> x.
			\end{cases} \]
			This implies
			\[\sup_{n\ge 1}\dfrac{1}{n}\sum_{i=1}^n \P(|X_{n,i}|>x)\ge \dfrac{1}{2}, \ \text{ for all } x\ge1.\]
			Thus by Corollary \ref{cor.character.for.stochastic.domination.Gut1}, there is no random variable $Y$ such that $\{X_{n,i},1\le i\le n,n\ge1\}$
			is stochastically dominated in the Ces\`aro sense by $Y$. Now, we have
			\[1<C_0=\sup_{n\ge 1}\sum_{i=1}^n a_{n,i}=\sup_{n\ge 2}\left(1+\dfrac{n-m_{n}}{n^2}\right)\le 2,\]
			and for $x\ge 1$ and $n\ge 1$,
			\begin{equation*}
				\sum_{i=1}^n a_{n,i}\P(|X_{n,i}|>x)=\sum_{i=m_n+1}^n\dfrac{1}{n^2}\P(|X_{n,i}|>x)=
				\begin{cases}
					0 & \text{ if } n\le x,\\
					\dfrac{n-m_n}{n^2}\le \dfrac{1}{n} & \text{ if } n> x.
				\end{cases}
			\end{equation*}
			Thus $\sum_{i=1}^n a_{n,i}\P(|X_{n,i}|>x)\to 0$ as $x\to\infty$.
			By Theorem \ref{thm.character.for.stochastic.domination.1}, 
			$\{X_{n,i},1\le i\le n,n\ge1\}$
			is $\{a_{n,i}\}$-stochastically dominated by a random variable $X$ with distribution function
			\[F(x)=1-\dfrac{1}{C_0}\sup_{n\ge 1}\sum_{i=1}^n a_{n,i}\P(|X_{n,i}|>x),\ x\in\R.\]
			This shows that the concept of $\{a_{n,i}\}$-stochastic domination is strictly weaker than the concept of stochastic domination in the Ces\`{a}ro sense.
		}
	\end{exm}
	
	Recall that a real-valued function $L(\cdot )$ is said to be \textit{slowly varying} (at infinity) if it is 
	a positive and measurable function on $[A,\infty)$ for some $A\ge0$, and for each $\lambda>0$,
	\begin{equation*}\label{rv01}
		\lim_{x\to\infty}\dfrac{L(\lambda x)}{L(x)}=1.
	\end{equation*}
	If $L(\cdot)$ is a slowly varying function, then
	there exists a slowly varying function $\tilde{L}(\cdot)$, unique up to an asymptotic equivalence, satisfying
	\begin{equation}\label{BGT1513}
		\lim_{x\to\infty}L(x)\tilde{L}\left(xL(x)\right)=1\ \text{ and } \lim_{x\to\infty}\tilde{L}(x)L\left(x\tilde{L}(x)\right)=1.
	\end{equation}
	The function $\tilde{L}(\cdot)$ is called the de Bruijn conjugate of $L(\cdot)$ (see Bingham et al. \cite[p. 29]{bingham1989regular}).
	For many ``nice'' slowly varying functions $L(\cdot)$,
	we can choose $\tilde{L}(x)=1/L(x)$. Especially, if $L(x)=(\log x)^\gamma$ or $L(x)=(\log\log x)^\gamma$ for some $\gamma\in\R$, 
	then $\tilde{L}(x)=1/L(x)$.

	Let $L(\cdot)$ be a slowly varying function and let $\alpha>0$. By using a suitable asymptotic equivalence version 
	(see Lemma 2.2 and Lemma 2.3 (i) in Anh et al. \cite{anh2021marcinkiewicz}), 
	we can firstly assume that $L(\cdot)$ is positive and
	differentiable on $[a,\infty)$, and $x^\alpha L(x)$ is strictly increasing on $[a,\infty)$  for some large $a$. Next, let
		$L_1(\cdot)$ be a slowly varying function satisfying $L_1(x)=L(a)x/a$ if $0\le x <a$ and $L_1(x)=L(x)$ if $x\ge a$ (i.e., 
		$L_1(0)=0$ with a linear growth to $L(a)$
		over $[0,a)$, and $L_1(x)\equiv L(x)$ on $[a,\infty)$).
		Then (i) $L_1(x)$ is continuous on $[0,\infty)$ and differentiable on $[a,\infty)$, and
		(ii) $x^\alpha L_1(x)$ is strictly increasing on $[0,\infty)$.
		In this paper, we will assume, without loss of generality, that these properties are fulfilled for the underlying slowly varying functions.

	The next theorem shows that bounded moment conditions 
	on an array of random variables $\{X_{n,i},1\le i\le k_n,n\ge1\}$
	with respect to weights $\{a_{n,i},1\le i\le k_n,n\ge1\}$ can accomplish $\{a_{n,i}\}$-stochastic domination.
	
	\begin{thm}\label{thm.sufficiency.for.stochastic.domination.2}
		Let $\{X_{n,i},1\le i\le k_n,n\ge1\}$ be an array of random variables,
		$\{a_{n,i},1\le i\le k_n,n\ge1\}$ an array of
		nonnegative
		real numbers satisfying \eqref{eq.weight01}. 
		Let $p>0$ and let $\nu$ be a fixed positive integer. Let $L(\cdot)$ be a 
		slowly varying function. If
		\begin{equation}\label{eq.stoch.domi.13}
			\sup_{n\ge 1}\sum_{i=1}^{k_n}a_{n,i}\E\left(|X_{n,i}|^pL(|X_{n,i}|)\log_{\nu}^{(2)}(|X_{n,i}|)\right)<\infty,
		\end{equation}
		then there exists a random variable $X$ 
		with distribution function 
		\begin{equation}\label{df11}
			F(x)=1-\dfrac{1}{C_0}\sup_{n\ge 1} \sum_{i=1}^{k_n}a_{n,i}\P(|X_{n,i}|>x),\ x\in\R
		\end{equation}
		such that $\{X_{n,i},1\le i\le k_n,n\ge1\}$ is $\{a_{n,i}\}$-stochastically dominated by $X$, and
		\begin{equation}\label{eq.stoch.domi.11}
			\E\left(|X|^pL(|X|)\right)<\infty.
		\end{equation}
	\end{thm}

	\begin{rem}
		{\rm 	A weaker version of Theorem \ref{thm.sufficiency.for.stochastic.domination.2}
			for stochastic domination (without the appearance of the slowly varying function $L(\cdot)$) 
			was proved by Rosalsky and Th\`{a}nh \cite{rosalsky2021note} (see Theorem 2.5 (ii) and (iii) in Rosalsky and Th\`{a}nh \cite{rosalsky2021note}).
			Typical examples of slowly varying functions $L(\cdot)$ for \eqref{eq.stoch.domi.13} are $L(x)\equiv1$ and
			$L(x)\equiv L_1(x)(\log_{\nu}^{(2)} (x))^{-1}$, where $L_1(\cdot)$
			is another slowly varying function.
			Theorem \ref{thm.sufficiency.for.stochastic.domination.2} is proved by employing an idea from Galambos and Seneta \cite{galambos1973regularly}.
			
	}\end{rem}
	
	Before proving Theorem \ref{thm.sufficiency.for.stochastic.domination.2}, we recall a simple result on the expectation of a nonnegative random variable, see Rosalsky and Th\`{a}nh \cite{rosalsky2021note}
	for a proof.
	
	\begin{lem}\label{lem.RT21}
		Let $h:[0,\infty)\to [0,\infty)$ be a measurable function with $h(0)=0$ which is
		bounded on $[0,A]$ and differentiable on $[A,\infty)$ for some $A\ge 0$.
		If $\xi$ is a nonnegative random variable, then
		\begin{equation*}
			\begin{split}
				\E(h(\xi))&=\E(h(\xi)\mathbf{1}(\xi\le A))+ h(A)+\int_{A}^\infty h'(x)\P(\xi>x)\dx x.
			\end{split}
		\end{equation*}
	\end{lem}
	
	\noindent\textit{Proof of Theorem \ref{thm.sufficiency.for.stochastic.domination.2}.}
	Set
	\[g(x)=x^pL(x)\log_{\nu}^{(2)}(x),\ \text{ and }h(x)=x^p L(x),\ x\ge 0.\]
	Since $\lim_{x\to\infty}g(x)=\infty$ and
	$g(\cdot)$ is strictly increasing on $[0,\infty)$ as we have assumed before, we have from Markov's inequality and \eqref{eq.stoch.domi.13} that
	\[0\le\lim_{x\to\infty}\sup_{n\ge 1} \sum_{i=1}^{k_n}a_{n,i}\P(|X_{n,i}|>x)\le \lim_{x\to\infty}\dfrac{1}{g(x)}\sup_{n\ge 1} \sum_{i=1}^{k_n}a_{n,i}\E(g(|X_{n,i}|))=0.\] 
	By Theorem \ref{thm.character.for.stochastic.domination.1}, the array $\{X_{n,i},1\le i\le k_n,n\ge1\}$ is 
	$\{a_{n,i}\}$-stochastically dominated by a random variable $X$
	with distribution function $F(\cdot)$ given in \eqref{df11}. 
	
	Next, we prove \eqref{eq.stoch.domi.11}. We firstly consider the case where the slowly varying function $L(\cdot)$ is differentiable on an infinite interval far enough from $0$, and
	\begin{equation}\label{approx.svf}
		\lim_{x\to \infty}\dfrac{xL'(x)}{L(x)}=0.
	\end{equation}
By \eqref{approx.svf}, there exists $B\ge 0$ such that 
	\[\left|\dfrac{xL'(x)}{L(x)}\right|\le \dfrac{p}{2},\ x>B.\]
It follows that
	\begin{equation}\label{eq.st.10}
		h'(x)=px^{p-1}L(x)+x^pL'(x)=x^{p-1}L(x)\left(p+\dfrac{xL'(x)}{L(x)}\right)\le \dfrac{3px^{p-1}L(x)}{2},\ x\ge B.
	\end{equation}
	Therefore, 
	there exists a constant $C_1$ such that
	\begin{equation*}
		\begin{split}
			\E(h(X))&=\E(h(X)\mathbf{1}(X\le B))+h(B)+\int_{B}^\infty h'(x)\P(X>x)\mathrm{d} x\\
			&\le C_1+\dfrac{3p}{2}\int_{B}^\infty x^{p-1}L(x)\P(X>x)\mathrm{d} x\\
			&= C_1+\dfrac{3p}{2C_0}\int_{B}^\infty x^{p-1}L(x)\sup_{n\ge1}\sum_{i=1}^{k_n}a_{n,i}\P(|X_{n,i}|>x)\mathrm{d} x\\
			&\le C_1+\dfrac{3p}{2C_0}\int_{B}^\infty \dfrac{1}{x\log_{\nu}^{(2)}(x)}\sup_{n\ge1}\sum_{i=1}^{k_n}a_{n,i}\E\left(g(|X_{n,i}|)\right) \mathrm{d} x\\
			&= C_1+\dfrac{3p}{2C_0}\sup_{n\ge1}\sum_{i=1}^{k_n}a_{n,i}\E\left(g(|X_{n,i}|)\right)\int_{B}^\infty \dfrac{\mathrm{d} x}{x\log_{\nu}^{(2)}(x)}<\infty,
		\end{split}
	\end{equation*}
	where we have applied Lemma \ref{lem.RT21} in the first equality,
	\eqref{eq.st.10} in the first inequality, Markov's inequality in the second inequality, and
	\eqref{eq.stoch.domi.13} in the last inequality. Thus we obtain \eqref{eq.stoch.domi.11} in this case.
	
	For general slowly varying function $L(\cdot)$, by a result from page 111 of Galambos and Seneta \cite{galambos1973regularly}, there exists
	a slowly varying function $L_1(\cdot)$ which is differentiable on $[B_1,\infty)$ for some $B_1$ large enough, and satisfies
	\begin{equation}\label{approx.svf3}
		\lim_{x\to \infty}\dfrac{L_1(x)}{L(x)}=1
	\end{equation}
	and
	\begin{equation*}\label{approx.svf5}
		\lim_{x\to \infty}\dfrac{xL_{1}'(x)}{L_{1}(x)}=0.
	\end{equation*}
	For $n\ge1$, $1\le i\le n$, we have from \eqref{approx.svf3} that for all $B_2$ large enough
	\begin{equation}\label{approx.svf7}
		\begin{split}
			&\E\left(|X_{n,i}|^pL_1(|X_{n,i}|)\log_{\nu}^{(2)}(|X_{n,i}|)\right)\\
			&\qquad=\E\left(|X_{n,i}|^pL_1(|X_{n,i}|)\log_{\nu}^{(2)}(|X_{n,i}|)\mathbf{1}(|X_{n,i}|\le B_2)\right)\\
			&\quad\qquad+ \E\left(|X_{n,i}|^pL_1(|X_{n,i}|)\log_{\nu}^{(2)}(|X_{n,i}|)\mathbf{1}(|X_{n,i}|> B_2)\right)\\
			&\qquad\le C_2+ 2\E\left(|X_{n,i}|^pL(|X_{n,i}|)\log_{\nu}^{(2)}(|X_{n,i}|)\mathbf{1}(|X_{n,i}|> B_2)\right),
		\end{split}
	\end{equation}
	where $C_2$ is a finite constant. Combining \eqref{eq.stoch.domi.13} and \eqref{approx.svf7} yields
	\begin{equation}\label{eq.stoch.domi.16}
		\sup_{n\ge 1}\sum_{i=1}^{k_n}a_{n,i}\E\left(|X_{n,i}|^pL_1(|X_{n,i}|)\log_{\nu}^{(2)}(|X_{n,i}|)\right)<\infty.
	\end{equation}
	Proceeding exactly the same manner as the first case with $L(\cdot)$ is replaced by $L_1(\cdot)$, we obtain \eqref{eq.stoch.domi.11}.
	The proof of the theorem is completed.
	\qed
	
	\vskip.1in
	The following corollary follows immediately from Theorem \ref{thm.sufficiency.for.stochastic.domination.2}.
	
	\begin{cor}\label{cor.character.for.stochastic.domination.Gut3}
		Let $\{X_{n,i},1\le i\le k_n,n\ge1\}$ be an array of random variables and let $L(\cdot)$ be a 
		slowly varying function. 
		Let $p>0$ and let $\nu$ be a fixed positive integer. If
		\begin{equation*}\label{eq.stoch.domi.Cesaro13}
			\sup_{n\ge 1}\dfrac{1}{k_n}\sum_{i=1}^{k_n}\E\left(|X_{n,i}|^pL(|X_{n,i}|)\log_{\nu}^{(2)} (|X_{n,i}|)\right)<\infty,
		\end{equation*}
		then there exists a random variable $X$ 
		with distribution function 
		\[			F(x)=1-\sup_{n\ge 1} \dfrac{1}{k_n}\sum_{i=1}^{k_n}\P(|X_{n,i}|>x),\ x\in\R\]
		such that $\{X_{n,i},1\le i\le k_n,n\ge1\}$ is stochastically dominated in the Ces\`{a}ro sense by $X$, and
		\begin{equation*}\label{eq.stoch.domi.Cesaro11}
			\E\left(|X|^pL(|X|)\right)<\infty.
		\end{equation*}
		
	\end{cor}
	
	\section{Relationships between $\{a_{n,i}\}$-stochastic domination and $\{a_{n,i}\}$-uniform integrability}\label{sec:relation}
	
	The concept of $\{a_{n,i}\}$-uniform integrability was introduced by Ord{\'o}{\~n}ez Cabrera \cite{ordonezcabrera1994convergence}.
	Let 
	$\{a_{n,i},1\le i\le k_n,n\ge1\}$ be an array of
	nonnegative
	real numbers satisfying \eqref{eq.weight01}. 
	An array $\{X_{n,i},1\le i\le k_n,n\ge1\}$ of random variables is said to
	be \textit{$\{a_{n,i}\}$-uniformly integrable} if
	\begin{equation*}\label{ui.eq03}
		\lim_{a\to\infty}\sum_{i=1}^{k_n}a_{n,i}\E(|X_{n,i}|\mathbf{1}(|X_{n,i}|>a))=0.
	\end{equation*}
	Similar to the classical characterization of the uniform integrability, it
	was proved by Ord{\'o}{\~n}ez Cabrera \cite{ordonezcabrera1994convergence} that an
	array of random variables
	$\{X_{n,i},1\le i\le k_n,n\ge1\}$ is
	$\{a_{n,i}\}$-uniformly integrable
	if and only if
	\[\sup_{n\ge 1}\sum_{i=1}^{k_n}a_{n,i}\E(|X_{n,i}|)<\infty\] 
	and for each $\varepsilon>0$, there exists $\delta>0$ such that whenever $\{A_{n,i},1\le i\le k_n,n\ge1\}$ is
	an array of events satisfying
	\[\sup_{n\ge 1}\sum_{i=1}^{k_n}a_{n,i}\P(A_{n,i})<\delta,\] 
	then
	\[\sup_{n\ge 1}\sum_{i=1}^{k_n}a_{n,i}\E(|X_{n,i}|\mathbf{1}(A_{n,i}))<\varepsilon. \] 
	If $a_{n,i}=1/k_n, 1\le i\le k_n,\ n\ge1,$	then it reduces to the concept of $\{X_{n,i},1\le i\le k_n,n\ge1\}$ being 
	\textit{uniformly integrable in the Ces\`{a}ro sense}
	which was introduced in \cite{chandra1989uniform}.
	The de La Vall\'{e}e--Poussin criterion for uniform integrability in the Ces\`{a}ro sense, and
	for $\{a_{n,i}\}$-uniform integrability was proved, respectively, by Chandra and Goswami \cite{chandra1992cesaro}, and
	Ord{\'o}{\~n}ez Cabrera \cite{ordonezcabrera1994convergence}. The former is a special case of the latter,
	which reads as follows:
	An array of random variables
	$\{X_{n,i},1\le i\le k_n,n\ge1\}$ is
	$\{a_{n,i}\}$-uniformly integrable
	if and only if there exists a measurable
	function $g: [0,\infty)\to [0,\infty)$ with $g(0) = 0$, $g(x)/x\to\infty$ as $x\to\infty$, and
	\[\sup_{n\ge 1} \sum_{i=1}^{k_n}a_{n,i}\E(g(|X_{n,i}|))<\infty.\] 
	Moreover, $g$ can be selected to be convex and such that $g(x)/x$ is nondecreasing.

	The next theorem establishes relationships between the concept of $\{a_{n,i}\}$-stochastic domination and the concept of $\{a_{n,i}\}$-uniform integrability.
	
	\begin{thm}\label{thm.sufficiency.for.stochastic.domination.5}
		Let $\{X_{n,i},1\le i\le k_n,n\ge1\}$ be an array of random variables and let
		$\{a_{n,i},1\le i\le k_n,n\ge1\}$ be an array of
		nonnegative
		real numbers satisfying \eqref{eq.weight01}. 
		Let $p>0$ and let $\tilde{L}(\cdot)$ be the Brujin conjugate of a 
		slowly varying function $L(\cdot)$.
		\begin{description}
			\item[(i)] If $\{X_{n,i},1\le i\le k_n,n\ge1\}$ is $\{a_{n,i}\}$-stochastically dominated by a random variable
			$X$ with $\E(|X|^pL(|X|^p))<\infty$, then $\{|X_{n,i}|^pL(|X_{n,i}|^p),1\le i\le k_n,n\ge1\}$ is $\{a_{n,i}\}$-uniformly integrable.		
			\item[(ii)] If $\{|X_{n,i}|^pL(|X_{n,i}|^p),1\le i\le k_n,n\ge1\}$ is $\{a_{n,i}\}$-uniformly integrable, then
			there exists a
			random variable $X$ with distribution function 
			\[F(x)=1-\dfrac{1}{C_0}\sup_{n\ge1} \sum_{i=1}^{k_n}a_{n,i}\P(|X_{n,i}|>x), x\in \R, \]
			such that $\{X_{n,i},1\le i\le k_n,n\ge1\}$ is $\{a_{n,i}\}$-stochastically dominated by $X$,
			\begin{equation}\label{eq.stoch.domi.14}
				\E\left(|X|^pL(|X|^p)(\log_{\nu}^{(2)} (|X|))^{-1}\right)<\infty \text{ for all fixed positive integer }\nu,
			\end{equation}
			and
			\begin{equation}\label{eq.stoch.domi.15}
				\lim_{x\to\infty}x\mathbb{P}\left(|X|>x^{1/p}\tilde{L}^{1/p}(x)\right)=0.
			\end{equation}
		\end{description}
	\end{thm}
	
	\begin{proof}
		Let $f(x)=x^pL(x^p),\ g(x)=x^{1/p}\tilde{L}^{1/p}(x),\ x\ge0$. 
		Recalling that we assume, without loss of generality, that $f$ and $g$ are strictly increasing on $[0,\infty)$.
		
		(i)
		Since $\E(|X|^pL(|X|^p))<\infty$, it follows from the classical
		de La Vall\'{e}e Poussin criterion for uniform integrability
		that there exists a continuous and strictly increasing function
		$h:[0,\infty)\to [0,\infty)$ with $h(0)=0$, $\lim_{x\to\infty}h(x)/x=\infty$,
		and $\E(h(|X|^pL(|X|^p)))<\infty.$ Since $f(x)$ is strictly increasing on $[0,\infty)$, 
		the $\{a_{n,i}\}$-stochastic domination assumption ensures that for all $n\ge1$,
		\[\sum_{i=1}^{k_n}a_{n,i}\P(|X_{n,i}|>f^{-1}(h^{-1}(x)))\le C_0 \P(|X|>f^{-1}(h^{-1}(x))),\ x\in\R\]
		or, equivalently,
		\[\sum_{i=1}^{k_n}a_{n,i}\P(|X_{n,i}|^pL(|X_{n,i}|^p)>h^{-1}(x))\le C_0 \P(|X|^pL(|X|^p)>h^{-1}(x)),\ x\in\R\]
		which, in turn, is equivalent to
		\[\sum_{i=1}^{k_n}a_{n,i}\P(h(|X_{n,i}|^pL(|X_{n,i}|^p))>x)\le C_0 \P(h(|X|^pL(|X|^p))>x),\ x\in\R.\]
		It follows that
		\begin{equation*}
			\begin{split}
				\sup_{n\ge 1} \sum_{i=1}^{k_n}a_{n,i}\E(h(|X_{n,i}|^pL(|X_{n,i}|^p)))&=\sup_{n\ge 1} \sum_{i=1}^{k_n}a_{n,i}\int_{0}^{\infty}\P(h(|X_{n,i}|^pL(|X_{n,i}|^p))>x)\dx x\\
				&\le C_0\int_{0}^{\infty}\P(h(|X|^pL(|X|^p))>x)\dx x\\
				&=C_0 \E(h(|X|^pL(|X|^p)))<\infty.
			\end{split}
		\end{equation*}
		By the de La Vall\'{e}e Poussin criterion for $\{a_{n,i}\}$-uniform integrability
		(Ord{\'o}{\~n}ez Cabrera \cite{ordonezcabrera1994convergence}),
		$\{|X_{n,i}|^pL(|X_{n,i}|^p),1\le i\le k_n,n\ge1\}$ is $\{a_{n,i}\}$-uniformly integrable.

		(ii) Since $\{|X_{n,i}|^pL(|X_{n,i}|^p),1\le i\le k_n,n\ge1\}$ is $\{a_{n,i}\}$-uniformly integrable,
		\[\sup_{n\ge 1}\sum_{i=1}^{k_n}a_{n,i}\E\left(|X_{n,i}|^pL(|X_{n,i}|^p)\right)<\infty,\] 
		and so by Theorem \ref{thm.sufficiency.for.stochastic.domination.2},
		$\{X_{n,i},1\le i\le k_n,n\ge1\}$ is $\{a_{n,i}\}$-stochastically dominated by a
		random variable $X$ with
		distribution function
		\[F(x)=1-\dfrac{1}{C_0}\sup_{n\ge 1} \sum_{i=1}^{k_n}a_{n,i}\P(|X_{n,i}|>x),\ x\in\R,\]
		and \eqref{eq.stoch.domi.14} holds.
		
		Finally, by using the de La Vall\'{e}e Poussin criterion for $\{a_{n,i}\}$-uniform integrability again, 
		there exists a nondecreasing function $h$ defined on $[0,\infty)$ with $h(0)=0$ such that
		\begin{equation}\label{eq.ui.domi.11}
			\lim_{x\to\infty}\dfrac{h(x)}{x}=\infty,
		\end{equation}
		and
		\begin{equation}\label{eq.ui.domi.12}
			\sup_{n\ge 1}\sum_{i=1}^{k_n}a_{n,i}\E(h(|X_{n,i}|^pL(|X_{n,i}|^p)))<\infty.
		\end{equation}
		By applying Lemma 2.1 in Anh et al. \cite{anh2021marcinkiewicz}, we have $f(g(x))/x\to 1$ as $x\to\infty$,
		and therefore
		\begin{equation}\label{ui.eq07}
			f(g(x))>x/2 \text{ for all large }x.
		\end{equation}
		We thus have from \eqref{eq.ui.domi.11}, \eqref{eq.ui.domi.12}, \eqref{ui.eq07} and Markov's inequality that
		\begin{equation*}
			\begin{split}
				\lim_{x\to\infty}x\mathbb{P}\left(|X|>g(x)\right)
				&=\dfrac{1}{C_0}\lim_{x\to\infty}x\sup_{n\ge 1}\sum_{i=1}^{k_n}a_{n,i}\mathbb{P}(|X_{n,i}|>g(x))\\
				&\le \dfrac{1}{C_0}\lim_{x\to\infty}x\sup_{n\ge 1}\sum_{i=1}^{k_n}a_{n,i}\mathbb{P}(f(|X_{n,i}|)\ge f(g(x)))\\
				&\le \dfrac{1}{C_0}\lim_{x\to\infty}x\sup_{n\ge 1}\sum_{i=1}^{k_n}a_{n,i}\mathbb{P}(f(|X_{n,i}|)\ge x/2)\\
				&\le \dfrac{1}{C_0}\lim_{x\to\infty}x\sup_{n\ge 1}\sum_{i=1}^{k_n}a_{n,i}\mathbb{P}(h(f(|X_{n,i}|))\ge h(x/2))\\
				&\le \dfrac{1}{C_0}\lim_{x\to\infty}x\sup_{n\ge 1}\sum_{i=1}^{k_n}a_{n,i}\dfrac{\E(h(f(|X_{n,i}|)))}{h(x/2)}\\
				&=\dfrac{2}{C_0}\sup_{n\ge 1}\sum_{i=1}^{k_n}a_{n,i}\E(h(f(|X_{n,i}|)))\lim_{x\to\infty}\dfrac{x/2}{h(x/2)}=0,
			\end{split}
		\end{equation*}
		thereby proving \eqref{eq.stoch.domi.15}.
		\qed
	\end{proof}
	
	The following corollary is a direct consequence of Theorem \ref{thm.sufficiency.for.stochastic.domination.5}.
	It plays an important role in establishing the weak laws of large numbers with general normalizing sequences
	under the Ces\`{a}ro uniform integrability condition in Section \ref{sec:proofs}. 
	
	\begin{cor}\label{cor.Cesaro.5}
		Let $\{X_{n,i},1\le i\le k_n,n\ge1\}$ be an array of random variables.
		Let $p>0$ and let $L(\cdot)$ be a 
		slowly varying function.
		\begin{description}
			\item[(i)] If $\{X_{n,i},1\le i\le k_n,n\ge1\}$ is stochastically dominated in the Ces\`{a}ro sense by a random variable
			$X$ with $\E(|X|^pL(|X|^p))<\infty$, then $\{|X_{n,i}|^pL(|X_{n,i}|^p),1\le i\le k_n,n\ge1\}$ is
			uniformly integrable in the Ces\`{a}ro sense.		
			\item[(ii)] If $\{|X_{n,i}|^pL(|X_{n,i}|^p),1\le i\le k_n,n\ge1\}$ is uniformly integrable in the Ces\`{a}ro sense, then
			there exists a
			random variable $X$ with distribution function 
			\[F(x)=1-\sup_{n\ge1} \dfrac{1}{k_n}\sum_{i=1}^{k_n}\P(|X_{n,i}|>x), x\in \R,\] 
			such that $\{X_{n,i},1\le i\le n,n\ge1\}$ is stochastically dominated in the Ces\`{a}ro sense by $X$,
			and \eqref{eq.stoch.domi.14} and \eqref{eq.stoch.domi.15} hold.
		\end{description}
	\end{cor}

	\section{Laws of large numbers for triangular arrays and proofs of Theorems \ref{thm.lln.main1} and \ref{thm.lln.main3}}\label{sec:proofs}
	
	In this section, we establish strong and weak laws of large numbers for triangular arrays of random variables.
	We say that a collection $\{X_i,1\le i\le N\}$ of
	random variables satisfies condition $(H)$ if for all $a>0$,
	there exists a constant $C$ such that
	\begin{equation}\label{eqH.condition}
		\E\left(\max_{1\le k\le n}\left|\sum_{i=m+1}^{m+k}\left(X_{i}^{(a)}-\E X_{i}^{(a)}\right)\right|\right)^2\le 
		C \sum_{i=m+1}^{m+n} \E(X_{i}^{(a)})^2,\ m\ge 0, n\ge 1, m+n\le N,
	\end{equation}
	where
	\[X_{i}^{(a)}=-a\mathbf{1}(X_i<-a)+X_i \mathbf{1}(|X_i|\le a)+a\mathbf{1}(X_i>a),\ 1\le i\le N.\]
	An infinite sequence of random variables $\{X_i,i\ge1\}$ is said to satisfy condition $(H)$ if
	every finite subsequence satisfies condition $(H)$.
	Many dependence structures meet this condition. For example, condition $(H)$ holds for negatively associated sequences, negatively superadditive dependent sequences,
	AANA sequences with the sequence of mixing coefficients is in 
	$\ell_2$ (the mixing coefficients $q_n,n\ge1$ satisfying $\sum_{n=1}^\infty q_{n}^2<\infty$).
	In
	Adler and Matu{\l}a \cite{adler2018exact}, the authors used a similar condition to establish 
	exact SLLNs (see Theorems 3.2 and 4.1 in \cite{adler2018exact}).

	Throughout this section, the symbol $C$ denotes a positive universal constant which is not necessarily the same in each appearance.
	We shall let the indices $k_n$ in the previous sections be $k_n\equiv n$.
	
	The following theorem establishes the rate of convergence in SLLN
	for maximal partial sums from triangular arrays of dependent 
	random variables under the Chandra--Ghosal-type condition (see Condition \eqref{chandra01.ext} below).
	
	\begin{thm}\label{thm.lln.array1}
 Let
$\{X_{n,i},1\le i\le n,n\ge1\}$ be an array of mean zero random variables 
		such that for each $n\ge1$ fixed, the collection $\{X_{n,i},1\le i\le n\}$ satisfies condition $(H)$, and let
		\[G(x)=\sup_{n\ge1}\dfrac{1}{n}\sum_{i=1}^{n}\P(|X_{n,i}|>x),\ x\in\R.\]
Let $L(\cdot)$ be a slowly varying function	and let $1\le p<2$.	When $p=1$, we further assume that $L(\cdot)$ is nondecreasing and $L(x)\ge 1$ for all $x\ge 0$.
		If 
		\begin{equation}\label{chandra01.ext}
			\int_{0}^{\infty}x^{p-1}L^p(x) G(x)\dx x<\infty,
		\end{equation} 
		then 
		\begin{equation}\label{main103}
			\sum_{n=1}^\infty n^{-1}\mathbb{P}\left(\max_{1\le k\le n}\left|\sum_{i=1}^k X_{n,i}\right|>\varepsilon n^{1/p}\tilde{L}(n^{1/p})\right)<\infty \text{ for all }\varepsilon>0,
		\end{equation}
		where $\tilde{L}(\cdot)$ is the Bruijn conjugate of $L(\cdot)$.
	\end{thm}
	
	\begin{proof}
		Since $L(\cdot)$ is a slowly varying function and $L(x)\ge 1$ for all $x\ge 0$ when $p=1$, it follows from \eqref{chandra01.ext} that $\lim_{x\to\infty} G(x)=0$.
		By Corollary \ref{cor.character.for.stochastic.domination.Gut1}, the array $\{X_{n,i},1\le i\le n,n\ge1\}$ 
		is stochastically dominated in the Ces\`{a}ro sense by a
		random variable $X$ with distribution function $F(x)=1-G(x),\ x\in \R.$
		Thus
		\begin{equation}\label{main105}
			G(x)=\sup_{n\ge 1}\dfrac{1}{n}\sum_{i=1}^n \P(|X_{n,i}|>x)=\mathbb{P}(|X|>x),\ x\in \R.
		\end{equation}
		Using the same arguments as in the proof of Theorem \ref{thm.sufficiency.for.stochastic.domination.2}, we can assume, without loss of generality, that
		the function $L(\cdot)$ satisfies
		\begin{equation}\label{sv015}
			\lim_{x\to \infty}\dfrac{xL'(x)}{L(x)}=0.
		\end{equation}
		Let $h(x)=x^pL^p(x), x\ge 0$. Then it follows from \eqref{sv015} that for all large $x$, 
\begin{equation}\label{sv016}
	h'(x)=px^{p-1}L^p(x)\left(1+\dfrac{xL'(x)}{L(x)}\right)\le \dfrac{3}{2}px^{p-1}L^p(x).
\end{equation}
		Applying Lemma \ref{lem.RT21}, it thus follows from \eqref{chandra01.ext}, \eqref{main105} and \eqref{sv016} that
		\begin{equation}\label{main107}
			\E(h(|X|))=\E\left(|X|^pL^p(|X|)\right)<\infty.
		\end{equation}
		We have proved that the array $\{X_{n,i},1\le i\le n,n\ge1\}$ is stochastically dominated in the Ces\`{a}ro sense by a random variable $X$ satisfying \eqref{main107}.
		For $n\ge1$, set $b_n= n^{1/p}\tilde{L}(n^{1/p})$,
		\[Y_{n,i}=-b_n \mathbf{1}(X_{n,i}<-b_n)+X_{n,i}\mathbf{1}(|X_{n,i}|\le b_n)+b_n \mathbf{1}(X_{n,i}>b_n),1\le i\le n,\]
		and
		\[S_{n,k}=\sum_{i=1}^k (Y_{n,i}-
		\E(Y_{n,i})),\ 1\le k\le n.\]
		We will now follow the proof of the implication ((i)$\Rightarrow$(ii))
		of Theorem 3.1 in Anh et al. \cite{anh2021marcinkiewicz}.
		Let $\varepsilon>0$ be arbitrary. For $n\ge 1$,
		\begin{equation}\label{main109}
			\begin{split}
				&\P\left(\max_{1\le k\le n}\left|\sum_{i=1}^k
				X_{n,i}\right|>\varepsilon  b_n\right)\le \P\left(\max_{1\le k\le n}|X_{n,k}|>b_n\right)
				+\P\left(\max_{1\le k\le n}\left|\sum_{i=1}^k Y_{n,i}\right|>\varepsilon  b_n\right)\\
				&\le \P\left(\max_{1\le k\le n}|X_{n,k}|>b_n\right)+\P\left(\max_{1\le k\le n}|S_{n,k}|>\varepsilon
				b_n- 
				\sum_{i=1}^n \left|\E(Y_{n,i})\right|\right).
			\end{split}
		\end{equation}
		Since the array $\{X_{n,i},1\le i\le n,n\ge1\}$ is stochastically dominated in the Ces\`{a}ro sense by a random variable $X$ satisfying \eqref{main107},
		we have from Proposition 2.6 in \cite{anh2021marcinkiewicz} that
		\begin{equation}\label{main111}
			\begin{split}
				\sum_{n= 1}^{\infty} n^{-1}\P\left(\max_{1\le k\le n}
				|X_{n,k}|>b_n\right)&\le \sum_{n=1}^{\infty} n^{-1}\sum_{k=1}^{n}\P
				\left(|X_{n,k}|>b_n\right)\\
				&\le \sum_{n=1}^{\infty} \P(|X|>b_n)<\infty.
			\end{split}
		\end{equation}
		For $n\ge 1$, it follows from the assumption $\E(X_{n,i})\equiv0$ and the Ces\`{a}ro stochastic domination condition that
		\begin{equation}\label{main113}
			\begin{split}
				\dfrac{\sum_{i=1}^{n}|\E(Y_{n,i})|}{b_n}&\le \dfrac{\sum_{i=1}^{n}\left(\left|\E(X_{n,i}\mathbf{1}(|X_{n,i}|\le b_n))\right|+b_n\P(|X_{n,i}|>b_n)\right)}{b_n}\\
				& = \dfrac{\sum_{i=1}^{n}\left(\left|\E(X_{n,i}\mathbf{1}(|X_{n,i}|> b_n))\right|+b_n\P(|X_{n,i}|>b_n)\right)}{b_n}\\
				& \le \dfrac{2\sum_{i=1}^{n}\E(|X_{n,i}|\mathbf{1}(|X_{n,i}|> b_n))}{b_n}\\
				& \le \dfrac{2n \E\left(|X|\mathbf{1}(|X|> b_n)\right)}{b_n}.
			\end{split}
		\end{equation}
		For $n$ large enough
		and for $\omega \in (|X|>b_n)$, we have (see (3.10) in \cite{anh2021marcinkiewicz})
		\begin{equation}\label{main115}
			\begin{split}
				\dfrac{n}{b_n}&\le  C|X(\omega)|^{p-1} L^p(|X(\omega)|).
			\end{split}
		\end{equation}
		Applying \eqref{main113}, \eqref{main115}, \eqref{main107} and the dominated convergence theorem, we have 
		\begin{equation}\label{main117}
			\begin{split}
				\dfrac{\sum_{i=1}^{n}|\E(Y_{n,i})|}{b_n}&
				\le C\E\left(|X|^\alpha L^p(|X|)\mathbf{1}\left(|X|> b_n\right)\right)\to 0 \text{ as } n\to \infty.
			\end{split}
		\end{equation}
		From \eqref{main109}, \eqref{main111} and \eqref{main117}, to obtain \eqref{main103}, it
		remains to show that
		\begin{equation}\label{main119}
			\sum_{n=1}^\infty 
			n^{-1}\P\left(\max_{1\le k\le n} |S_{n,k}|>b_n\varepsilon /2 \right)<\infty.
		\end{equation}
		Applying Markov's inequality, condition ($H$), and the Ces\`{a}ro stochastic domination condition yields
		\begin{equation}\label{main121}
			\begin{split}
				&\sum_{n=1}^{\infty} \dfrac{1}{n}\P\left(\max_{1\le k\le n}
				|S_{n,k}|>b_n\varepsilon /2 \right)
				\le \sum_{n=1}^{\infty}\dfrac{4 }{\varepsilon ^2 nb_{n}^2}
				\E\left(\max_{1\le k\le n}|S_{n,k}|\right)^2\\
				& \le \sum_{n=1}^{\infty}\dfrac{C}{nb_{n}^2}
				\sum_{i=1}^n\E\left(Y_{n,i}^2\right)\\
				& \le C\sum_{n=1}^{\infty}\dfrac{\left(\E \left(X^2\mathbf{1}(|X|\le b_n)\right)+b_{n}^2 \P(|X|>b_n)\right)}{b_{n}^{2}}
				\\
				&= C\sum_{n=1}^{\infty}\left(\P(|X|>b_n)+\dfrac{\E(X^2\mathbf{1}(|X|\le b_n))}{b_{n}^{2}}\right).
			\end{split}
		\end{equation}
Using the last four lines of (3.13) of \cite{anh2021marcinkiewicz} and (3.14) of \cite{anh2021marcinkiewicz}, we have
		\begin{equation}\label{main122}
	\begin{split}
		\sum_{n=1}^{\infty}\left(\P(|X|>b_n)+\dfrac{\E(X^2\mathbf{1}(|X|\le b_n))}{b_{n}^{2}}\right)\le C+ C\E(|X|^pL^p(|X|)).
	\end{split}
\end{equation}
Combining \eqref{main121}, \eqref{main122}, and \eqref{main107} yields \eqref{main119}.
		\qed
	\end{proof}
	
	The following corollary establishes rate of
	convergence in a Marcinkiewicz--Zygmund-type SLLN for
	arrays of random variables under a uniformly bounded
	moment condition.
	
	\begin{cor}\label{cor.bounded.moment01}
		Let $1\le p<2$ and let $\nu$ be a fixed positive integer. Let $\{X_{n,i},1\le i\le n,n\ge1\}$ be an array of mean zero random variables 
		such that for each $n\ge1$ fixed, the collection $\{X_{n,i},1\le i\le n\}$ satisfies condition $(H)$. If
		\begin{equation}\label{bounded.moment03}
			\sup_{n\ge 1}\dfrac{1}{n}\sum_{i=1}^{n}\E\left(|X_{n,i}|^pL^p(|X_{n,i}|)\log_{\nu}^{(2)} (|X_{n,i}|)\right)<\infty,
		\end{equation}	
		then \eqref{main103} is obtained.
	\end{cor}
	
	\begin{proof}
		By applying Corollary \ref{cor.character.for.stochastic.domination.Gut3}, we have from 
		\eqref{bounded.moment03} that the array $\{X_{n,i},1\le i\le n,n\ge1\}$
		is stochastically dominated in the Ces\`{a}ro sense by a random variable $X$ with $\E(|X|^pL^p(|X|))<\infty$, that is, \eqref{chandra01.ext} is satisfied. 
		Applying Theorem \ref{thm.lln.array1}, we obtain \eqref{main103}.
		\qed
	\end{proof}
	
	The moment condition \eqref{bounded.moment03} is almost optimal. The following example shows that there exists
	an array of random variables $\{X_{n,i},1\le i\le n,n\ge1\}$ such that
	\begin{equation}\label{bounded.moment04}
		\sup_{n\ge 1}\dfrac{1}{n}\sum_{i=1}^{n}\E\left(|X_{n,i}|^p \log_{\nu} (|X_{n,i}|)\right)<\infty
	\end{equation}	
	for every fixed positive integer $\nu$, but \eqref{main103} fails with $\tilde{L}(x)\equiv L(x)\equiv1$.
	
	\begin{exm}\label{exm.optimal.moment.for.comp.conv}
		{\rm 
			Let $\nu$ be an arbitrary fixed positive integer and let $1\le p<2$. Let
			$\{X_n,n\ge1\}$ be a sequence of independent random variables with
			\[\P(X_{n}=0)=1-\dfrac{1}{n\log_{\nu}(n)},\  \P\left(X_{n}=\pm(n+1)^{1/p}\right)=\dfrac{1}{2n\log_{\nu}(n)}, n\ge1\]
			and let $X_{n,i}=X_i,1\le i\le n,n\ge1.$
			Then \eqref{bounded.moment04} is satisfied, and
			\begin{equation}\label{exm25}
				\begin{split}
					\sum_{n=1}^\infty \P(|X_{n}|>n^{1/p})&=	\sum_{n=1}^\infty\dfrac{1}{n\log_{\nu}(n)}=\infty.
				\end{split}
			\end{equation}
			If \eqref{main103} (with $\tilde{L}(x)\equiv L(x)\equiv1$) holds, then
			\begin{equation*}\label{exm27}
				\sum_{n=1}^\infty\dfrac{1}{n}\mathbb{P}\left(\max_{1\le j\le n}\left|\sum_{i=1}^j X_{i}\right|>\varepsilon n^{1/p}\right)<\infty \text{ for all }\varepsilon>0.
			\end{equation*}
			This implies that
			\[	\lim_{n\to\infty}\dfrac{\sum_{i=1}^nX_i}{n^{1/p}}=0 \text{ a.s,}\]
			and thus
			\begin{equation}\label{exm29}
				\lim_{n\to\infty}\dfrac{X_{n}}{n^{1/p}}=0 \text{ a.s.}
			\end{equation}
			Applying the Borel--Cantelli lemma, we have from \eqref{exm29} that
			\begin{equation*}\label{eq.exm23}
				\sum_{n=1}^\infty \P(|X_{n}|>n^{1/p})<\infty
			\end{equation*}
			contradicting \eqref{exm25}. Therefore, \eqref{main103} (with $\tilde{L}(x)\equiv L(x)\equiv1$) must fail.
			
		}

	\end{exm}

	If we consider sequences of random variables instead of triangular arrays, we obtain the following Marcinkiewicz--Zygmund-type SLLN with general normalizing 
	sequences.
	
	\begin{cor}\label{cor.MZ01}
Let $\{X_{n},n\ge1\}$ be a sequence of mean zero random variables 
		satisfying condition $(H)$. Let
		$L(\cdot)$ be a slowly varying function and let $1\le p<2$.
		When $p=1$, we further assume that $L(\cdot)$ is nondecreasing and $L(x)\ge 1$ for all $x\ge 0$.
		If \begin{equation}\label{chandra02.ext}
			\int_{0}^{\infty}x^{p-1}L^p(x)\left(\sup_{n\ge1}\dfrac{1}{n}\sum_{i=1}^n \P(|X_i|>x)\right) \dx x<\infty,
		\end{equation} 
		then 
		\begin{equation}\label{MZ11}
			\lim_{n\to\infty}\dfrac{\sum_{i=1}^nX_i}{n^{1/p}\tilde{L}(n^{1/p})}=0 \text{ a.s.,}
		\end{equation}
		where $\tilde{L}(\cdot)$ is the Bruijn conjugate of $L(\cdot)$.
	\end{cor}
	
	\begin{proof}
		Set 
		\[X_{n,i}=
		X_i,1\le i\le n,n\ge1.\]
		Then \eqref{chandra02.ext} coincides with \eqref{chandra01.ext}.
		Applying Theorem \ref{thm.lln.array1}, we have
		\begin{equation}\label{main302}
			\sum_{n=1}^\infty\dfrac{1}{n}\mathbb{P}\left(\max_{1\le j\le n}\left|\sum_{i=1}^j X_{i}\right|>\varepsilon n^{1/p}\tilde{L}(n^{1/p})\right)<\infty \text{ for all }\varepsilon>0.
		\end{equation}
		The Marcinkiewicz-Zygmund-type SLLN \eqref{MZ11} follows from \eqref{main302}. \qed
	\end{proof}
	
	\noindent\textit{Proof of Theorem \ref{thm.lln.main1}.}
	Since $\{X_n,n\ge1\}$ is a sequence of AANA random variables
	with the sequence of mixing coefficients is in $\ell_2$, it satisfies condition ($H$) (see Lemmas 2.1 and 2.2 of Ko et al. \cite{ko2005hajeck}).
	Thus Theorem \ref{thm.lln.main1} follows from Corollary \ref{cor.MZ01} by taking $L(x)\equiv1$. \qed
	\vskip.1in
	
	The following theorem is a significant extension of Theorem \ref{thm.lln.main3}. It establishes a WLLN for weighted sums from arrays of random variables.
	
	\begin{thm}\label{thm.lln.array2}
		Let $\{X_{n,i},1\le i\le n,n\ge1\}$ be an array of random variables. Let $G(\cdot)$ be
as in Theorem \ref{thm.lln.array1} and let
		$\{b_n,n\ge1\}$ be a nondecreasing sequence of positive real numbers satisfying \eqref{wl01}.
		Let $\{c_{n,i},1\le i\le n\}$ be an array of nonnegative real numbers satisfying
		\begin{equation}\label{wl.weight3}
			0<A_n:=\sum_{i=1}^n c_{n,i}\le C n,\ n\ge 1
		\end{equation}
		and let
		\[\hat{G}(x)=\sup_{n\ge1}\sum_{i=1}^{n}a_{n,i}\P(|X_{n,i}|>x),\ x\in\R,\]
		where $a_{n,i}=A_{n}^{-1}c_{n,i}, \ 1\le i\le n, n\ge 1.$
		If
		\begin{equation}\label{wl.array03}
			\lim_{k\to\infty}kG(b_k)=0\ \text{ and }\ \lim_{k\to\infty}k\hat{G}(b_k)=0,
		\end{equation}
		then the WLLN 
		\begin{equation}\label{wl.array05}
			\dfrac{1}{b_n}\max_{j\le n}\left|\sum_{i=1}^j c_{n,i}X_{n,i}\right|\overset{\P}{\to} 0 \text{ as }n\to\infty
		\end{equation}
		is obtained.
	\end{thm}
	
	\begin{rem}{\rm 
			It is clear that in the unweighted case, i.e., $c_{n,i}\equiv1$, then $\hat{G}(x)\equiv G(x)$.
It is also easy to see that if $\{X_{n,i},1\le i\le n,n\ge1\}$ is stochastically dominated by a random variable $X$ with $\lim_{k\to\infty}k\P(|X|>b_k)=0$, then
			both halves of \eqref{wl.array03} are fulfilled.}
	\end{rem}
	
	At the first look, the second half of \eqref{wl.array03} may seem to be  a technical condition. However, the following example shows that it cannot be dispensed with.
	
	\begin{exm}{\rm 
			Let $0<p<1$, $b_n=n^{1/p},\ n\ge1$ and let $\{X_{n,i},1\le i\le n,n\ge1\}$ be
			an array of random variables such that
			\begin{equation*}
				\mathbb{P}(X_{n,i}=-1)=\mathbb{P}(X_{n,i}=1)=1/2\ \text{ for }1\le i<n,n\ge 2, 
			\end{equation*}
			and 
			\begin{equation*}
				\mathbb{P}(X_{n,n}=-n^{1/p}\log^{-1/p}(n))=\mathbb{P}(X_{n,n}=n^{1/p}\log^{-1/p}(n))=1/2\ \text{ for }n\ge 1.
			\end{equation*}
			Let $\{c_{n,i},1\le i\le n,n\ge1\}$ be an array of real numbers such that
			\[c_{n,i}=0 \text{ for }1\le i<n, n\ge2, \text{ and }c_{n,n}=n \text{ for }n\ge1\]
			and let
			\[A_n=\sum_{i=1}^n c_{n,i},\ a_{n,i}=\dfrac{c_{n,i}}{A_n},1\le i\le n,n\ge1. \]
			Then \eqref{wl.weight3} is satisfied since $A_n\equiv n$.
			Let 
			\[G(x)=\sup_{n\ge1}\dfrac{1}{n}\sum_{i=1}^{n}\P(|X_{n,i}|>x),\ x\in\R,\] 
			and
			\[\hat{G}(x)=\sup_{n\ge1}\sum_{i=1}^{n}a_{n,i}\P(|X_{n,i}|>x),\ x\in\R.\]
			For $n\ge 1$, we have
			\begin{equation*}
				\begin{split}
					\dfrac{1}{n}\sum_{i=1}^{n}\E\left(|X_{n,i}|^p\log(|X_{n,i}|)\right)&\le 1+\dfrac{1}{n}\E(|X_{n,n}|^p\log(|X_{n,n}|))\\
					&=1+\dfrac{1}{\log(n)}\left(\dfrac{\log(n)-\log(\log(n))}{p}\right)\\
					&\le 1+\dfrac{1}{p}<\infty.
				\end{split}
			\end{equation*}
			Therefore, $\{|X_{n,i}|^p,1\le i\le n,n\ge1\}$ is uniformly integrable in the Ces\`{a}ro sense by the de La Vall\'{e}e Poussin criterion for the Ces\`{a}ro uniform integrability.
			Then by Corollary \ref{cor.Cesaro.5}, $\{X_{n,i},1\le i\le n,n\ge1\}$ is
			stochastically dominated in the Ces\`{a}ro sense by a random variable $X$ 
			with distribution function
			\begin{equation*}
				F(x)=1-\sup_{n\ge 1}\dfrac{1}{n}\sum_{i=1}^n\mathbb{P}(|X_{n,i}|>x)=1-G(x),\ x\in\R,
			\end{equation*}
			and the first half of \eqref{wl.array03} (with $b_n\equiv n^{1/p}$) is satisfied. 
			However, the second half of \eqref{wl.array03} (with $b_n\equiv n^{1/p}$) fails since
			\[\hat{G}(x)=\sup_{n\ge 1}\sum_{i=1}^n a_{n,i}\P(|X_{n,i}|>x)=\sup_{n\ge 1}\P(|X_{n,n}|>x)=1 \text{ for all }x\in\R.\]
			For $n\ge1$, we have with probability $1$,
			\begin{equation*}
				\begin{split}
					\dfrac{1}{b_n}\max_{j\le n}\left|\sum_{i=1}^j c_{n,i}X_{n,i}\right|&=\dfrac{1}{n^{1/p}}c_{n,n}|X_{n,n}|\\
					&=\dfrac{n}{\log^{1/p}(n)}\to\infty
				\end{split}
			\end{equation*}
			therefore the WLLN \eqref{wl.array05} also fails.}
	\end{exm}

	\noindent\textit{Proof of Theorem \ref{thm.lln.array2}.}
	From \eqref{wl01}, we have $b_n\to\infty$ (see \eqref{wl137} below). Since $G(x)$ and $\hat{G}(x)$ are nonincreasing, it follows from \eqref{wl.array03} that
	$\lim_{x\to\infty} G(x)=0$ and $\lim_{x\to\infty} \hat{G}(x)=0.$
	By Corollary \ref{cor.character.for.stochastic.domination.Gut1}, 
	$\{X_{n,i},1\le i\le n,n\ge1\}$ is stochastically dominated in the Ces\`{a}ro sense by a random variable $X$,
	and by Theorem \ref{thm.character.for.stochastic.domination.1},
	$\{X_{n,i},1\le i\le n,n\ge1\}$ is $\{a_{n,i}\}$-stochastically dominated by a random variable $Y$.
	The distribution functions of $X$ and $Y$, respectively, are
	\[F_X(x)=1-G(x)\ \text{ and }\ F_Y(x)=1-\hat{G}(x),\ x\in \R.\]
	Thus
	\begin{equation}\label{wl125}
		\sup_{n\ge 1}\dfrac{1}{n}\sum_{i=1}^n \P(|X_{n,i}|>x)=\mathbb{P}(|X|>x),\ x\in \R,
	\end{equation}
	and
	\begin{equation}\label{wl126}
		\sup_{n\ge 1}\sum_{i=1}^n a_{n,i}\P(|X_{n,i}|>x)=\mathbb{P}(|Y|>x),\ x\in \R,
	\end{equation}
	and so \eqref{wl.array03} becomes
	\begin{equation}\label{wl127}
		\lim_{k\to\infty}k\mathbb{P}(|X|>b_k)=0\ \text{ and }\ \lim_{k\to\infty}k\mathbb{P}(|Y|>b_k)=0.
	\end{equation}
	For $n\ge 1$, set
	\[Y_{n,i}=X_{n,i}\mathbf{1}(|X_{n,i}|\le b_n),\ 1\le i\le n.\]
	We first verify that
	\begin{equation}\label{wl133}
		\dfrac{\max_{1\le j\le n}\left|\sum_{i=1}^j c_{n,i}\left( X_{n,i} - Y_{n,i} \right)\right|}{b_n}\overset{\mathbb{P}}{\to} 0 \text { as } n \rightarrow \infty.
	\end{equation}
	To see this, let $\varepsilon > 0 $ be arbitrary. Then we have from \eqref{wl125} and the first half of \eqref{wl127} that
	\begin{equation*}
		\begin{split}
			\mathbb{P}\left(\max_{1\le j\le n}\left|\sum_{i=1}^j c_{n,i}\left( X_{n,i} - Y_{n,i} \right)\right|>b_n\varepsilon\right)
			&\le \mathbb{P}\left(\bigcup_{i=1}^n(X_{n,i}\not=Y_{n,i})\right)\\
			&\le \sum_{i=1}^n \mathbb{P}(|X_{n,i}|>b_n)\\
			&\le n\mathbb{P}(|X|>b_n)\to 0 \text{ as } n\to\infty
		\end{split}
	\end{equation*}
	thereby proving $\eqref{wl133}$.
	
	Next, it will be shown that 
	\begin{equation}\label{wl136}
		\dfrac{\max_{1\le j\le n}\left|\sum_{i=1}^j c_{n,i}Y_{n,i}\right|}{b_n}\overset{\mathbb{P}}{\to} 0 \text { as } n \rightarrow \infty.
	\end{equation}
	To accomplish this, we first recall that \eqref{wl01} implies (see Remark 2.4 (i) in Boukhari \cite{boukhari2021weak})
	\begin{equation}\label{wl137}
		\dfrac{n}{b_n}\to 0 \text{ as }n\to\infty.
	\end{equation}
	Set $b_0=0$. Again, let $\varepsilon > 0 $ be arbitrary. Then
	\begin{equation}\label{wl138}
		\begin{split}
			\mathbb{P}\left(\max_{1\le j\le n}\left|\sum_{i=1}^j c_{n,i}Y_{n,i}\right|>b_n\varepsilon\right)
			&\le \dfrac{1}{b_n\varepsilon}\mathbb{E}\left(\max_{1\le j\le n}\left|\sum_{i=1}^j c_{n,i}Y_{n,i}\right|\right)\\
			&\le \dfrac{1}{b_n\varepsilon}\sum_{i=1}^n c_{n,i}\E(|Y_{n,i}|)\\
			&\le \dfrac{1}{b_n\varepsilon}\sum_{i=1}^n c_{n,i} \int_{0}^{b_n}\mathbb{P}(|X_{n,i}|>x)\dx x\\
			&= \dfrac{A_n}{b_n\varepsilon}\sum_{i=1}^n a_{n,i} \int_{0}^{b_n}\mathbb{P}(|X_{n,i}|>x)\dx x\\
			&\le \dfrac{A_n}{b_n\varepsilon}\int_{0}^{b_n}\mathbb{P}(|Y|>x)\dx x\\
			&= \dfrac{A_n}{b_n\varepsilon}\sum_{k=1}^n\int_{b_{k-1}}^{b_k}\mathbb{P}(|Y|>x)\dx x\\
			&\le \dfrac{Cn}{b_n\varepsilon}\sum_{k=1}^n\dfrac{b_k-b_{k-1}}{k}k\mathbb{P}(|Y|>b_{k-1}),
		\end{split}
	\end{equation}
	where we have applied Markov's inequality in the first inequality, \eqref{wl126} in the fourth inequality, and \eqref{wl.weight3}
	in the last inequality. Now when $n\ge 2$,
	\begin{equation*}
		\begin{split}
			\dfrac{n}{b_n}\sum_{k=1}^n\dfrac{b_k-b_{k-1}}{k}&=\dfrac{n}{b_n}\left(\sum_{k=1}^{n-1}\dfrac{b_k}{k(k+1)}+\dfrac{b_n}{n}\right)\\
			&\le \dfrac{n}{b_n}\left(\sum_{k=1}^{n-1}\dfrac{b_k}{k^2}+\dfrac{b_n}{n}\right)\\
			&\le C \ \text{ (by \ref{wl01})},
		\end{split}
	\end{equation*}
	for all fixed $k$,
	\begin{equation*}
		\dfrac{n}{b_n}\left(\dfrac{b_{k}-b_{k-1}}{k}\right)\to 0 \text{ as } n\to\infty\ \text{ (by \eqref{wl137})},
	\end{equation*}
	and for $k\ge 2$,
	\begin{equation*}
		k\mathbb{P}(|Y|>b_{k-1})\le 2(k-1)\mathbb{P}(|Y|>b_{k-1}) \to 0 \text{ as } k\to\infty\ \text{ (by the second half of \ref{wl127})}.
	\end{equation*}
	Thus by the Toeplitz lemma
	\begin{equation*}
		\dfrac{n}{b_n}\sum_{k=1}^n\dfrac{b_k-b_{k-1}}{k}k\mathbb{P}(|Y|>b_{k-1})\to 0 \text{ as } n\to\infty
	\end{equation*}
	and \eqref{wl136}  then follows from \eqref{wl138}. Combining \eqref{wl133} and \eqref{wl136} yields \eqref{wl.array05}.\qed
	
	\vskip.1in
	\noindent\textit{Proof of Theorem \ref{thm.lln.main3}.}
	Set 
	\[c_{n,i}=1,\  X_{n,i}=X_i,1\le i\le n,n\ge1.\]
	Then \eqref{wl03} coincides with \eqref{wl.array03}.
	Theorem \ref{thm.lln.main3} follows from Theorem \ref{thm.lln.array2}.\qed
	
	\vskip.1in
	We will now
	present an example to illustrate Theorem \ref{thm.lln.main3}.
	This example shows that 
	for $0<p<1$ and $b_n=n^{1/p}$,
	there exists a sequence of random variables $\{X_n,n\ge1\}$ with no stochastically dominating random variable and
	condition \eqref{wl03} is satisfied. In this example, we also show that \eqref{chandra03} (with $0<p<1$) 
	holds but \eqref{chandra01} (with $0<p<1$) does not.
	
	\begin{exm} {\rm 
			Let $0<p<1$, $b_n=n^{1/p},\ n\ge1$ and let $\{X_n,n\ge1\}$ be
			a sequence of random variables such that
			\begin{equation}\label{exm01}
				\mathbb{P}(X_n=-1)=\mathbb{P}(X_n=1)=1/2\ \text{ for }n\not=2^m,\ m\ge0,
			\end{equation}
			and 
			\begin{equation}\label{exm02}
				\mathbb{P}\left(X_{2^m}=-2^{m/p}/m^{1/p}\right)=\mathbb{P}\left(X_{2^m}=2^{m/p}/m^{1/p}\right)=1/2 \text{ for } m\ge0.
			\end{equation}
			Then
			\[\sup_{n\ge 1}\mathbb{P}(|X_n|>x)= \sup_{m\ge 1} \mathbb{P}(|X_{2^m}|>x)=1 \text{ for all }x\ge 1.\]
			Thus there is no random variable $X$ such that the sequence $\{X_n,n\ge1\}$ is stochastically dominated by $X$,
			and so we cannot apply Theorem 2.1 of Boukhari \cite{boukhari2021weak}.

			Now, for $n\ge1$, let $m\ge0$ be such that $2^m\le n<2^{m+1}$. Then
			\begin{equation}
				\begin{split}
					\dfrac{1}{n}\sum_{i=1}^{n}\E\left(|X_i|^p\log(|X_i|)\right)&\le 1+\dfrac{1}{2^m}\sum_{i=0}^m\E\left(|X_{2^i}|^p\log(|X_{2^i}|)\right)\\
					&=1+\dfrac{1}{2^m}\sum_{i=0}^m\dfrac{2^i}{i}\left(\dfrac{i-\log(i)}{p}\right)\\
					&\le 1+\dfrac{2}{p}<\infty.
				\end{split}
			\end{equation}
			Therefore, $\{|X_n|^p,n\ge1\}$ is uniformly integrable in the Ces\`{a}ro sense by the de La Vall\'{e}e Poussin criterion for the Ces\`{a}ro uniform integrability.
			Then by Corollary \ref{cor.Cesaro.5}, the sequence $\{X_n,n\ge1\}$ is
			stochastically dominated in the Ces\`{a}ro sense by a random variable $X$ 
			with distribution function
			\begin{equation}\label{exm03}
				F(x)=1-\sup_{n\ge 1}\dfrac{1}{n}\sum_{i=1}^n\mathbb{P}(|X_i|>x):=1-G(x),\ x\in\R,
			\end{equation}
			and condition \eqref{wl03} (with $b_n\equiv n^{1/p}$) is satisfied. It thus follows from Theorem
			\ref{thm.lln.main3} that the WLLN \eqref{wl05} holds.
			
			Finally, it is clear that $\mathbb{P}(|X_n|^p>n)=0$ for all $n\ge 1$ so
			that \eqref{chandra03} holds. We will show that \eqref{chandra01} 
			(with $0<p<1$) is not satisfied. To see this, for $x\ge 1$, let $n_x$ be
			the smallest integer such that
			\[\dfrac{2^{n_x}}{n_{x}}>x^p.\]
			Then for $x\ge1$, 
			\begin{equation}\label{exm05}
				\dfrac{2^{n_x}}{n_{x}}>x^p\ge \dfrac{2^{n_x-1}}{n_{x}-1}
			\end{equation}
			and it follows from \eqref{exm05} that there exists $\varepsilon_0>0$ such that
			\begin{equation}\label{exm07}
				\dfrac{1}{2^{n_x}}>\dfrac{\varepsilon_0}{x^p\log(x)} \ \text{ for all large }x.
			\end{equation}
			Combining \eqref{exm01}--\eqref{exm03} and \eqref{exm07}, we have
			\begin{equation}\label{exm09}
				G(x)=\sup_{n\ge 1}\dfrac{1}{n}\sum_{i=1}^n\mathbb{P}(|X_i|>x)\ge \dfrac{1}{2^{n_x}}\sum_{i=1}^{2^{n_x}}
				\mathbb{P}(|X_i|>x)=\dfrac{1}{2^{n_x}}\ge \dfrac{\varepsilon_0}{x^p\log(x)} \text{ for all large } x
			\end{equation}
			and it follows from \eqref{exm09} that 
			\eqref{chandra01} 
			(with $0<p<1$) fails. Thus we cannot apply Remark 3 of Chandra and Ghosal \cite{chandra1996extensions} to obtain the Marcinkiewicz--Zygmund
			SLLN for the case $0<p<1$.}
	\end{exm}

	Now, we discuss about the normalizing sequences in the WLLN in Theorem \ref{thm.lln.main3}. We note that for $0< p<2$ and $b_n\equiv n^{1/p}$, 
	\eqref{wl01} is fulfilled if $0<p<1$ but
	it fails to hold if $1\le p< 2$. For the case where $1\le p<2$, we also have to require some dependence structures to obtain the WLLN (see Boukhari \cite{boukhari2021weak}
	for a counterexample).
	Kruglov \cite{kruglov2011generalization} established
	a Kolmogorov--Feller-type WLLN for sequences of negatively associated identically distributed 
	random variables with normalizing sequences $b_n,n\ge1$ satisfying
	\begin{equation}\label{wl.array51}
		\sum_{i=1}^{n}\dfrac{b_{i}^2}{i^2}=O\left(\dfrac{b_{n}^2}{n}\right).
	\end{equation}
	It was observed by Kruglov \cite{kruglov2011generalization} that \eqref{wl.array51} holds for the case $b_n\equiv n^{1/p}L(n)$,
	where $1\le p<2$ and $L(\cdot)$ is a slowly varying function.
	The following theorem appears to be new even when the underlying random variables are independent.
	It also	extends the sufficient part of Theorem 1 of Kruglov \cite{kruglov2011generalization}. 
	The proof of Theorem \ref{thm.lln.array3} can be obtained by proceeding in a similar manner as that of Theorem \ref{thm.lln.array2}:
	We firstly use the nondecreasing
	truncation as in the proof of Theorem \ref{thm.lln.array1}, and 
	use the maximal inequality \eqref{eqH.condition} instead of
	the triangular inequality (in the place of the second inequality in \eqref{wl138}), and then change the other places accordingly. 
	We leave the details to the interested reader.
	
	\begin{thm}\label{thm.lln.array3}
		Let $\{X_{n,i},1\le i\le n,n\ge1\}$ be an array of random variables
		such that for each $n\ge1$ fixed, the collection $\{X_{n,i},1\le i\le n\}$ satisfies condition $(H)$ and let $G(\cdot)$
be as in Theorem \ref{thm.lln.array1}. Let
		$\{b_n,n\ge1\}$ be a nondecreasing sequence of positive real numbers satisfying 
		\eqref{wl.array51} and let $\{c_{n,i},1\le i\le n,n\ge1\}$ be an array of nonnegative real numbers satisfying
		\begin{equation}\label{wl.weight13}
			0<A_n:=\sum_{i=1}^n c_{n,i}^2\le C n,\ n\ge 1.
		\end{equation}
Let
		\[\hat{G}(x)=\sup_{n\ge1}\sum_{i=1}^{n}a_{n,i}\P(|X_{n,i}|>x),\ x\in\R,\]
		where $a_{n,i}=A_{n}^{-1}c_{n,i}^2, \ 1\le i\le n, n\ge 1.$
		If
		\begin{equation}\label{wl.array53}
			\lim_{k\to\infty}kG(b_k)=0\ \text{ and }\ \lim_{k\to\infty}k\hat{G}(b_k)=0,
		\end{equation}
		then the WLLN 
		\begin{equation*}
			\dfrac{1}{b_n}\max_{j\le n}\left|\sum_{i=1}^j c_{n,i}\left(X_{n,i}-\E(X_{n,i}\mathbf{1}(|X_{n,i}|\le b_n))\right)\right|\overset{\P}{\to} 0 \text{ as }n\to\infty
		\end{equation*}
		is obtained.
	\end{thm}

	We now apply Corollary \ref{cor.Cesaro.5} and Theorem \ref{thm.lln.array3} to obtain a WLLN for arrays of random variables
	under the Ces\`{a}ro uniform integrability condition.
	For simplicity and since it is just meant to be an illustration, we only consider the unweighted case, i.e., the case where $c_{n,i}\equiv1$.
	\begin{cor}\label{cor.wl.array6}
		Let $1\le p<2$ and let $\{X_{n,i},1\le i\le n,n\ge1\}$ be an array of random variables
		such that for each $n\ge1$ fixed, the collection $\{X_{n,i},1\le i\le n\}$ satisfies condition $(H)$. Let $G(\cdot)$
		be as in Theorem \ref{thm.lln.array1} and $L(\cdot)$ be a slowly varying function. Let $\tilde{L}(\cdot)$ be
 the Bruijin conjugate of $L(\cdot)$. In the case $p=1$, we further assume that $L(x)$ is nondecreasing and $L(x)\ge 1$ for all $x\ge 0$.
		If $\{|X_{n,i}|^pL(|X_{n,i}|^p),1\le i\le n, n\ge1\}$ is uniformly integrable in the Ces\`{a}ro sense, that is,
		\begin{equation}\label{wl.array56}
			\lim_{a\to\infty}\sup_{n\ge 1}\dfrac{1}{n}\sum_{i=1}^n  \E\left(|X_{n,i}|^pL(|X_{n,i}|^p)\mathbf{1}(|X_{n,i}|>a)\right)=0,
		\end{equation}
		then the WLLN 
		\begin{equation}\label{wl.array57}
			\dfrac{1}{n^{1/p}\tilde{L}^{1/p}(n)}\max_{j\le n}\left|\sum_{i=1}^j \left(X_{n,i}-\E(X_{n,i})\right)\right|\overset{\mathbb{P}}{\to} 0 \text{ as }n\to\infty
		\end{equation} is obtained.
	\end{cor}
	
	\begin{proof}
		In Theorem \ref{thm.lln.array3}, if we choose $c_{n,i}\equiv1$, then \eqref{wl.weight13} is automatic, and $\hat{G}(x)\equiv G(x)$. By Corollary \ref{cor.Cesaro.5},
		it follows from \eqref{wl.array56} that \eqref{wl.array53} holds with $b_n\equiv n^{1/p}\tilde{L}^{1/p}(n)$. 
		Applying Theorem \ref{thm.lln.array3}, we obtain
		\begin{equation*}
			\dfrac{1}{b_n}\max_{j\le n}\left|\sum_{i=1}^j \left(X_{n,i}-\E(X_{n,i}\mathbf{1}(|X_{n,i}|\le b_n))\right)\right|\overset{\mathbb{P}}{\to} 0 \text{ as }n\to\infty.
		\end{equation*}
		To obtain \eqref{wl.array57}, it remains to show that
		\begin{equation}\label{wl.array61}
			\dfrac{1}{b_n}\max_{j\le n}\left|\sum_{i=1}^j \E(X_{n,i}\mathbf{1}(|X_{n,i}|> b_n))\right| \to  0 \text{ as }n\to\infty.
		\end{equation}
		Since the function $x^{p-1}L(x)$ is nondecreasing, we have
		\begin{equation}\label{wl.array65}
			\begin{split}
				&\dfrac{1}{b_n}\max_{j\le n}\left|\sum_{i=1}^j \E(X_{n,i}\mathbf{1}(|X_{n,i}|> b_n))\right| \le \dfrac{1}{b_n}\sum_{i=1}^n \E(|X_{n,i}|\mathbf{1}(|X_{n,i}|> b_n))\\
				&\le \dfrac{1}{b_n}\sum_{i=1}^n \dfrac{\E(|X_{n,i}|^pL(|X_{n,i}|^p)\mathbf{1}(|X_{n,i}|> b_n))}{b_{n}^{p-1}L(b_{n}^p)}\\
				&=\left(\dfrac{1}{\tilde{L}(n)L\left(n\tilde{L}(n)\right)}\right)\dfrac{1}{n}\sum_{i=1}^n \E(|X_{n,i}|^pL(|X_{n,i}|^p)\mathbf{1}(|X_{n,i}|> b_n)).
			\end{split}
		\end{equation}
		Now, from the second half of \eqref{BGT1513} we have $\lim_{n\to\infty} \tilde{L}(n)L\left(n\tilde{L}(n)\right)=1$, and from \eqref{wl.array56} we have
		\[\lim_{n\to\infty}\dfrac{1}{n}\sum_{i=1}^n \E(|X_{n,i}|^pL(|X_{n,i}|^p)\mathbf{1}(|X_{n,i}|> b_n))=0.\]
		Therefore \eqref{wl.array61} follows from \eqref{wl.array65}. The proof of the corollary is completed. \qed
	\end{proof}
	
	\section{Conclusions and open problems}\label{sec:open}
	
	In Section \ref{sec:proofs}, our results on the concept of $\{a_{n,i}\}$-stochastic domination
	are applied to obtain the WLLNs for weighted sums. The results on the Ces\`{a}ro stochastic domination case
	are applied to obtain rate of convergence in the SLLN with general normalizing sequences under the Chandra--Ghosal-type condition, 
	and these results help us to remove an assumption of a SLLN established by Chandra and Ghosal \cite{chandra1996extensions}. 
	The results on the concept of $\{a_{n,i}\}$-stochastic domination may also be useful in proving weighted SLLNs of Chandra and Ghosal \cite{chandra1996strong}
	as we will describe as follows.
	
	Let $1\le p<2$ and
	let $\{a_n,n\ge1\}$ be a sequence of positive real numbers with
	\[A_n:=\sum_{i=1}^n a_{i}\to\infty\ \text{ as }n\to\infty.\]
	Let $\{X_n,n\ge1\}$ be a sequence of mean zero random variables which satisfies suitable dependence conditions. 
	Let $G(x)$ be as in Theorem \ref{thm.chandra}
	and let
	\[\tilde{G}(x)=\sup_{n\ge 1}\left(\sum_{i=1}^n a_{i}^{1/p}\right)^{-1}\sum_{i=1}^n a_{i}^{1/p}\mathbb{P}(|X_i|>x),x\in\R.\]
	Chandra and Ghosal \cite{chandra1996strong} considered the following three conditions (see (2.14)--(2.16) in \cite{chandra1996strong}):
	\begin{equation}\label{chandra11}
		\int_{0}^{\infty}x^{p-1}G(x)\dx x<\infty,
	\end{equation}
	\begin{equation}\label{chandra13}
		\int_{0}^{\infty}x^{p-1}\tilde{G}(x)\dx x<\infty,
	\end{equation}
	and
	\begin{equation}\label{chandra15}
		\sum_{n= 1}^{\infty}\mathbb{P}(|X_n|^p>A_n/a_n)<\infty.
	\end{equation}	
	Chandra and Ghosal (see Theorems 2.6 and 2.7 in \cite{chandra1996strong}) proved that if
	\eqref{chandra11}, \eqref{chandra13} and \eqref{chandra15} are all satisfied, then the weighted Marcinkiewicz--Zygmund SLLN
	\begin{equation*}\label{chandra17}
		\dfrac{\sum_{i=1}^n a_{i}^{1/p} X_i}{A_{n}^{1/p}} \to 0 \text{ a.s. as }n\to\infty
	\end{equation*}	
	is obtained. In view of Theorems \ref{thm.lln.main1} and \ref{thm.lln.array2}, we state an open problem as to whether or not 
	the Chandra and Ghosal result mentioned above still holds without
	Condition \eqref{chandra15}.
	
	For $n\ge1$, let
	\[a_{n,i}=\left(\sum_{i=1}^n a_{i}^{1/p}\right)^{-1}a_{i}^{1/p}, \text{ }1\le i\le n.\]
	Since $\tilde{G}(x)$ is nonincreasing, it follows from \eqref{chandra13} that $\lim_{x\to \infty}\tilde{G}(x)=0$.
	By Theorem \ref{thm.character.for.stochastic.domination.1}, $\{X_n,n\ge1\}$ is $\{a_{n,i}\}$-stochastically dominated by
	a random variable $X$ with distribution function $F(x)=1-\tilde{G}(x)$, and \eqref{chandra13} becomes
	$\E(|X|^p)<\infty$. In view of the proof of Theorem \ref{thm.lln.array2}, the results on the concept of $\{a_{n,i}\}$-stochastic domination established in Sections \ref{sec:stochastic_domination}
	and \ref{sec:relation} 
	may help in answering the above open problem.
	
	Finally, we present an open problem concerning Corollary \ref{cor.wl.array6}.
	For the case where $L(x)\equiv\tilde{L}(x)\equiv 1$, we can obtain convergence in mean of order $p$ in \eqref{wl.array57} (see, e.g.,
	Theorem 1 in \cite{chandra1989uniform},
	Theorem 4 in \cite{ordonezcabrera1994convergence}, Theorem 2.1 in \cite{thanh2005lp}). 
	However, the methods in \cite{chandra1989uniform,ordonezcabrera1994convergence,thanh2005lp}
	do not seem to work for general slowly varying function $L(\cdot)$, even with assumption that the underlying random variables are independent. 
	It is an open problem as to whether or not convergence in mean of order $p$ prevails in
	the conclusion \eqref{wl.array57}.
	
	\vskip.2in	
\noindent	\textbf{Conflict of interest:} The author has no conflict of interest.
	\vskip.2in	
\noindent	\textbf{Funding:} The author did not receive support from any organization for this work.

	\bibliographystyle{spmpsci}      
	\bibliography{mybib}
	
\end{document}